\documentclass[a4paper,11pt, oneside]{amsart}
\usepackage[english]{babel}
\usepackage{amsthm}
\usepackage{amssymb}
\usepackage{amsfonts}
\usepackage{amsmath}
\usepackage{amsthm}
\usepackage[all]{xy}
\usepackage{geometry}
\usepackage{ae,aecompl}
\usepackage{eurosym}
\usepackage{verbatim}
\usepackage{mathrsfs}
\usepackage{caption}
\usepackage{url}
\usepackage{tikz}
\usepackage{pifont}
\usepackage{hyperref}
\usepackage{calligra}
 \usepackage{mathtools}
\usepackage{tikz-cd}

\theoremstyle{plain}

\newtheorem{thm}{Theorem}[section]

\newtheorem{lem}[thm]{Lemma}
\newtheorem{prop}[thm]{Proposition}
\newtheorem{cor}[thm]{Corollary}

\theoremstyle{definition}

\newtheorem{eg}[thm]{Example}
\newtheorem{defn}[thm]{Definition}
\newtheorem{rem}[thm]{Remark}
\newtheorem{remark}[thm]{Remark}

\date{}

\newcommand\bit{\begin{itemize}}
\newcommand\eit{\end{itemize}}
\newcommand\bet{\begin{enumerate}}
\newcommand\eet{\end{enumerate}}
\newcommand\ed{\end{document}}

\DeclareFontFamily{U}{mathx}{\hyphenchar\font45}
\DeclareFontShape{U}{mathx}{m}{n}{
      <5> <6> <7> <8> <9> <10>
      <10.95> <12> <14.4> <17.28> <20.74> <24.88>
      mathx10
      }{}
\DeclareSymbolFont{mathx}{U}{mathx}{m}{n}
\DeclareFontSubstitution{U}{mathx}{m}{n}
\DeclareMathAccent{\widecheck}{0}{mathx}{"71}
\DeclareMathAccent{\wideparen}{0}{mathx}{"75}

\renewcommand{\a}{\alpha}

\newcommand{\e}{\varepsilon}

\renewcommand{\k}{\kappa}

\newcommand\s{\sigma}

\newcommand\w{\omega}

\newcommand\Om{\Omega}

\newcommand\del{\partial}
\newcommand\adel{\ol{\partial}}
\newcommand\DEL{\Delta}

\newcommand\bC{{\mathbb C}}

\newcommand\bR{{\mathbb R}}
\newcommand\bZ{{\mathbb Z}}

\newcommand\B{{\mathcal B}}

\newcommand\F{{\mathcal F}}
\renewcommand\H{\mathcal{H}}
\renewcommand{\O}{\mathcal{O}}

\newcommand\exd{\mathrm{d}}
\newcommand\dom{\mathrm{dom}}

\newcommand\haar{\mathrm{\bf h}}

\newcommand\unit{\mathrm{U}}

\newcommand\id{\mathrm{id}}

\newcommand\oby{\otimes}

\newcommand\wed{\wedge}
\newcommand\sseq{\subseteq}

\def\qbinom#1#2{\ensuremath{\left[\kern-.3em\left[\genfrac{}{}{0pt}{}{#1}{#2}\right]\kern-.3em\right]_q}}

\newcommand\ol{\overline}

\newcommand\EE{{\mathcal E}}
\newcommand\FF{{\mathcal F}}
\renewcommand\H{\mathcal{H}}
\newcommand{\OO}{\mathcal{O}}

\usepackage{tikz}
\usetikzlibrary{decorations.pathreplacing}

\usepackage{ytableau}

\author{Biswarup Das}

\address{Instytut Matematyczny, Uniwersytet Wroc\l{}awski, pl.Grunwaldzki 2/4, 50-384 Wroc\l{}aw, Poland}
\email{biswarup.das@math.uni.wroc.pl}

\author[R. \'O Buachalla]{R\'eamonn \'O Buachalla}
\address{Mathematical Institute of Charles University, Sokolovsk\'a 83, Prague, Czech Republic} \email{obuachalla@karlin.mff.cuni.cz}

\author{Petr Somberg}
\address{Mathematical Institute of Charles University, Sokolovsk\'a 83, Prague, Czech Republic} \email{somberg@karlin.mff.cuni.cz}

\title[{\bf Compact  Quantum Homogeneous K\"ahler Spaces}]{{\bf  Compact  Quantum Homogeneous K\"ahler Spaces
}}

\thanks{R\'{O}B is supported by the GA\v{C}R/NCN grant \emph{Quantum Geometric Representation Theory and Non- commutative Fibrations} 24-11728K, and acknowledges support from COST Action 21109 CaLISTA, supported by COST (European Cooperation in Science and Technology) and HORIZON- MSCA-2022-SE-01-01 CaLIGOLA. PS was supported by the GA\v{C}R grant GA22-00091S.}

\begin{document}

\maketitle

\begin{abstract}
Noncommutative K\"ahler structures provide an algebraic framework for studying noncommutative complex geometry on quantum homogeneous spaces. In this paper, we introduce the notion of a \emph{compact quantum homogeneous K\"ahler space} which gives a natural set of compatibility conditions between covariant K\"ahler  structures and Woronowicz's theory of  compact quantum groups. Each such object admits a Hilbert space completion possessing a remarkably rich yet tractable structure.  The analytic behaviour of the associated Dolbeault--Dirac operators is moulded by the complex geometry of the underlying calculus. In particular, twisting the Dolbeault--Dirac operator by a negative Hermitian holomorphic module is shown to give a Fredholm operator if and only if the top anti-holomorphic cohomology group is finite-dimensional. In this case, the operator's index coincides with the twisted holomorphic Euler characteristic of the underlying noncommutative complex structure. The irreducible quantum flag manifolds, endowed with their Heckenberger--Kolb calculi, are presented as motivating examples.
\end{abstract}

\section{Introduction}

Since the emergence of  quantum groups in the 1980s, a central role in their presentation and development has been played by the theory of operator algebras. We mention in particular  Woronowicz's seminal notion of a compact quantum group  \cite{WoroCQPGs,WoroDC}. There exists, however,  a stark contrast between the development of the noncommutative topological and the noncommutative differential geometric aspects of the theory. For example, for Drinfeld--Jimbo quantum groups, their $C^*$-algebraic  $K$-theory has long been known to be the same as for their  classical counterparts \cite{Nagy}. By contrast, the unbounded formulation of $K$-homology, which is to say Connes and Moscovici's theory of spectral triples, remains very poorly understood.  Indeed, despite  a  large  number  of  very important  contributions  over  the  last  thirty  years,  there  is  still  no consensus on how to construct a spectral triple for $\O_q(SU_2)$, probably the most fundamental example of a quantum group.

There does, however, exist a long standing algebraic approach to the construction of \mbox{$q$-deformed} differential operators for quantum groups based on the theory of covariant differential calculi. This has its origins in the work of Woronowicz \cite{WoroDC}, with steady advances made in the following decades by many others, most notably Majid \cite{BeggsMajid:Leabh}. As has become increasingly clear in recent years, this approach is particularly suited to the study of quantum flag manifolds, quantum homogeneous spaces which $q$-deform the coordinate rings of the classical flag manifolds  $G/L_S$. These quantum spaces are distinguished by being braided commutative algebra objects in the braided monoidal category of $\O_q(G)$-comodules, and have a geometric structure much closer to the classical situation than quantum groups themselves. This is demonstrated by the existence of an essentially unique $q$-deformed de Rham complex for the \emph{irreducible} quantum flag manifolds, as  shown by Heckenberger and Kolb in their seminal series of papers \cite{HK, HKdR, HKBGG}. This makes the quantum flag manifolds a far more tractable starting point than quantum groups for investigating $q$-deformed noncommutative geometry.


The classical flag manifolds are compact connected homogeneous K\"ahler manifolds, providing us with a rich store of  classical geometric structures to exploit. Motivated by this, the notion of a noncommutative Hermitian  structure was introduced in \cite{MMF3} to provide a framework in which to study the noncommutative geometry of the quantum flag manifolds. Many of the fundamental results of  Hermitian and K\"ahler geometry  follow from the existence of such a structure, providing powerful tools with which to study the underlying calculus. The existence of a K\"ahler structure was verified for  the Heckenberger--Kolb calculus of quantum projective space in \cite{MMF3}. This result was later extended  by Matassa  \cite{MarcoConj} to every Heckenberger--Kolb calculus, for all but a finite number of values of $q$. Moreover, K\"ahler structures have been discovered in the setting of holomorphic \'etale groupoids \cite{JJGos}, and for finite graphs \cite{FiniteGraphKahler}, promising a much wider domain of application than initially expected.


In this paper we build on this rich algebraic and geometric structure to produce a  theory of bounded and unbounded differential operators acting on square integrable forms. 
We do so in the novel framework of compact quantum homogeneous Hermitian spaces (CQH-Hermitian spaces) which detail a natural set of compatibility conditions between covariant Hermitian structures and Woronowicz's theory of compact quantum groups. Every CQH-Hermitian space is shown to have a naturally associated Hilbert space completion. Moreover, much of the theory of Hermitian structures carries over to the square integrable setting, giving almost complex and Lefschetz decompositions, as well as a bounded representation of $\frak{sl}_2$. The de Rham, holomorphic, and anti-holomorphic differentials also behave very well with respect to completion.  All three Dirac operators $D_{\del}, D_{\adel}$, and $D_{\exd}$ are  essentially self-adjoint, giving us access to powerful analytic machinery such as functional calculus. The spectral and index theoretic properties of these operators are intimately connected with the curvature and cohomology of the underlying calculus. Moreover, these operators are highly amenable to applications of concepts and structures from classical complex geometry. As shown in \textsection \ref{section:TFFFred}, twisting the anti-holomorphic Dolbeault--Dirac operator of a CQH-K\"ahler space by a negative (anti-ample) relative Hopf module produces a Fredholm operator if and only if the top anti-holomorphic cohomology group is finite-dimensional.  Just as in the classical case, Hodge theory then implies that the index of the twisted operator  is  given by the twisted anti-holomorphic Euler characteristic. This invariant can be determined by geometric means. In particular, for positive modules, it follows from the Kodaira vanishing theorem for noncommutative K\"ahler structures that all higher cohomologies vanish, meaning that the  index is concentrated in degree zero. The case of negative modules follows analogously through an application of noncommutative Serre duality \cite[\textsection 6.2]{OSV}. 

 
There exists a large literature dealing with the analytic properties of Dolbeault--Dirac operators over quantum flag manifolds $\O_q(G/L_S)$. For example, see the pioneering papers \cite{ DSPodles, SISSACPn, RoberOwczarek}, or the review of such constructions  in \cite[\textsection 1]{DOS1}. This paper is inspired by many of the observations and results in these works. Indeed, our principal motivation for introducing CQH-K\"ahler spaces is to provide a robust formal framework in which to study the analytic properties of Dolbeault--Dirac operators over $\O_q(G/L_S)$.

\subsection{Summary of the Paper}

The paper is organised as follows: In \textsection 2 we recall from \cite{MMF3} the necessary basics of Hermitian and K\"ahler structures, and Hermitian holomorphic modules.

In \textsection 3, we introduce the notion of compact quantum homogeneous  Hermitian space $(B \subseteq A, \Omega^{\bullet}, \Omega^{(\bullet,\bullet)}, \sigma)$, and its twist by an Hermitian module $\F$ . We then use Takeuchi's categorical equivalence to show boundedness of morphisms, giving us a bounded representation of $\frak{sl}_2$ on $L^2(\Omega^{\bullet} \otimes_B \F)$. Moreover, we prove boundedness of multiplication operators, and conclude boundedness of the commutators $[D_{\adel_{\F}},b]$, for all $b \in B$.

In \textsection 4 we discuss closability and  essential self-adjointness for our Dirac and Laplace operators.

In \textsection 5 we show that twisting by a negative module produces a  Fredholm operator if and only if the top anti-holomorphic cohomology group of $\F$ is finite-dimensional. 


In \textsection 6  we recall our  motivating family of examples, the irreducible quantum flag manifolds $\O_q(G/L_S)$ endowed with their Heckenberger--Kolb calculi. We produce a family of Dolbeault--Dirac Fredholm operators for each $\O_q(G/L_S)$ through twisting by a negative line module. Moreover, we give an explicit presentation of the operator index in terms of the Weyl dimension formula.

We finish with three  appendices. The first recalls the theory of compact quantum groups algebras, the second recalls elementary material on unbounded operators, and the third discusses the relationship with the theory of spectral triples.

\subsubsection*{Acknowledgements:} The authors would like to thank Karen Strung, Branimir \'{C}a\'{c}i\'c, Elmar Wagner,  Fredy D\'iaz Garc\'ia, Andrey Krutov, Simon Brain, Adam Rennie, Paolo Saracco, Kenny De Commer, Matthias Fischmann, Adam--Christiaan van Roosmalen, Jan \v{S}t\!'ov\'i\v{c}ek, Zhaoting Wei, Thomas Weber, and Myriam Mahaman for many useful discussions during the preparation of this paper.  The second author would like to thank IMPAN Wroc\l{}aw for hosting him in November 2017, and would also  like to thank the Institute for Mathematics, Astrophysics and Particle Physics, Radboud University, Nijmegen for hosting him in the winter of 2017 and 2018.

\section{Preliminaries on Hermitian Structures}

We recall the basic definitions and results for Hermitian, and K\"ahler structures over differential $*$-calculi. For a detailed discussion of differential $*$-calculi in general, see  \cite{BeggsMajid:Leabh}. For a more detailed introduction to K\"ahler structures see \cite{MMF3}. For a presentation of classical complex and K\"ahler geometry see \cite{HUY}. All algebras are unital and over $\mathbb{C}$, and all unadorned tensor products are over $\mathbb{C}$.

\subsection{Differential Calculi and Complex Structures}

A {\em differential calculus}, or \emph{dc}, is a differential graded algebra $\big(\Omega^\bullet \cong \bigoplus_{k \in \mathbb{Z}_{\geq 0}} \Omega^k, \exd\big)$ generated as an algebra by the elements $a, \exd b$, for $a,b \in \Omega^0$.  A dc \emph{over} an algebra $B$ is dc such that $\Omega^0 = B$. We call an element of $\Omega^k$ a $k$-form. A {\em $*$-differential calculus}, or a \emph{$*$-dc}, over a $*$-algebra $B$ is a dc over $B$ such that the \mbox{$*$-map} of $B$ extends to a conjugate linear involutive map $*:\Omega^\bullet \to \Omega^\bullet$ satisfying $\exd(\w^*) = (\exd \w)^*$, and 
$
\big(\w \wed \nu\big)^*  =  (-1)^{kl} \nu^* \wed \w^*,  \text{ for all } \w \in \Omega^k, \, \nu \in \Omega^l. 
$
See \cite[\textsection 1]{BeggsMajid:Leabh} for a more detailed discussion of differential calculi.

An {\em almost complex structure} $\Omega^{(\bullet,\bullet)}$, for a  $*$-dc  $(\Omega^{\bullet},\exd)$, is an $\mathbb{Z}^2_{\geq 0}$-algebra grading 
of $\Omega^{\bullet}$ such that
\begin{align*} 
\Omega^k = \bigoplus_{a+b = k} \Omega^{(a,b)}, & & \big(\Omega^{(a,b)}\big)^* = \Omega^{(b,a)}, & & \textrm{ for all } (a,b) \in \mathbb{Z}^2_{\geq 0}. 
\end{align*}
We call an element of $\Omega^{(a,b)}$ an $(a,b)$-form. If the exterior derivative decomposes into a sum $\exd = \del + \adel$, for $\del$ a (necessarily unique) degree $(1,0)$-map, and $\adel$ a (necessarily unique) degree $(0,1)$-map, then we say that $\Omega^{(\bullet,\bullet)}$ is a \emph{complex structure}. It follows that we have a double complex. The {\em opposite} complex structure of a complex structure $\Omega^{(\bullet,\bullet)}$ is the \mbox{$\mathbb{Z}^2_{\geq 0}$}-algebra grading  $\overline{\Om}^{(\bullet,\bullet)}$, defined by $\ol{\Om}^{(a,b)} := \Omega^{(b,a)}$, for $(a,b) \in \mathbb{Z}_{\geq 0}^2$. See \cite[\textsection 1]{BeggsMajid:Leabh} or \cite{BS, KLvSPodles, MMF2} for a more detailed discussion of complex structures.

\subsection{Hermitian and K\"ahler Structures} \label{subsection:HandKS}

We recall the definition of an Hermitian structure, as introduced in \cite[\textsection 4]{MMF3}, which generalises the properties of the fundamental form of an Hermitian metric.  A dc is said to be of {\em total degree} $m \in \mathbb{Z}_{\geq 0}$ if $\Omega^m \neq 0$, and $\Omega^{k} = 0$, for all $k > m$.

\begin{defn} An {\em Hermitian structure} $(\Omega^{(\bullet,\bullet)}, \s)$ for a $*$-dc $\Omega^{\bullet}$, of even total degree $2n$,  is a pair  consisting of  a complex structure  $\Omega^{(\bullet,\bullet)}$, and  a closed central real $(1,1)$-form $\sigma$, called the {\em Hermitian form}, such that, with respect to the {\em Lefschetz operator}
$L_{\sigma}:\Omega^\bullet \to \Omega^\bullet$, defined by $\omega \mapsto \sigma \wed \omega$, isomorphisms are given by
\begin{align} \label{eqn:Liso}
L_{\sigma}^{n-k}: \Omega^{k} \to  \Omega^{2n-k}, & & \text{ for all } k = 0, \dots, n-1.
\end{align}
A \emph{K\"ahler structure} is an Hermitian structure $(\Omega^{(\bullet,\bullet)}, \kappa)$ satisfying $\exd \kappa = 0$, and in this case we refer to $\kappa$ as the \emph{K\"ahler form}.
\end{defn} 

For $L_{\sigma}$ the Lefschetz operator of an Hermitian structure, we denote
\begin{align*}
P^{k} : = \begin{cases} 
      \{\a \in \Omega^{k} \,|\, L_{\sigma}^{n-k+1}(\a) = 0\}, &  \text{ ~ if } k \leq n,\\
      0 & \text{ ~ if } k > n.
   \end{cases}
\end{align*}
An element of $P^{\bullet} := \oplus_{k=0}^n P^k$ is called a {\em primitive form}. As established in \cite[Proposition 4.3]{MMF3}, a $B$-bimodule decomposition of $\Omega^{k}$, for all $k \in \mathbb{Z}_{\geq 0}$,  is given by
\begin{align*}
\Omega^{k} \cong \bigoplus_{j \in \mathbb{Z}_{\geq 0}} L_{\sigma}^j\big(P^{k-2j}\big).
\end{align*}
We call this the \emph{Lefschetz decomposition} of $\Omega^{\bullet}$.

In classical Hermitian geometry, the Hodge map of an Hermitian metric is related to the associated Lefschetz decomposition through the Weil formula (see \cite[Th\'eor\`eme 1.2]{Weil} or \cite[Proposition 1.2.31]{HUY}). In \cite[Definition 4.11]{MMF3} a noncommutative  generalisation of the Weil formula was used to define a noncommutative Hodge map $\ast_{\sigma}$ for any noncommutative Hermitian structure. The {\em metric} associated to an Hermitian structure $\big(\Omega^{(\bullet,\bullet)},\s\big)$ is the unique sesquilinear form $g_\s:\Omega^\bullet \times \Omega^\bullet \to B$  for which $g_{\sigma}\big(\Omega^k, \Omega^l\big) = 0$, for all $k \neq l$, and for which 
\begin{align*}
g_\s(\w, \nu) =  \ast_\s\big(\!\ast_{\sigma}(\w^*) \wedge \nu \big), & &  \textrm{ for all } \w,\nu \in \Omega^k.
\end{align*}
An important fact is that $g_{\sigma}$ is \emph{conjugate symmetric} that is,
\begin{align*}
g_{\sigma}(\w, \nu) = g_{\sigma}(\nu, \w)^*, & & \text{ for all } \w, \nu \in \Omega^\bullet,
\end{align*}
see \cite[Corollary 5.3]{MMF3} for details. For a $*$-algebra $B$, we consider the \emph{cone of positive elements} 
\[ 
B_{\geq  0} := \Big\{\sum_i b_i^* b_i  \,|\, b_i \in B \Big\}\!.
\]
We denote the non-zero positive elements of $B$ by $B_{>0} := B_{\geq 0} \backslash \{0\}$. We say that an Hermitian structure $(\Omega^{(\bullet,\bullet)}, \sigma)$ is \emph{positive definite} if the associated metric $g_{\sigma}$ is \emph{positive definite}, which is to say, if $g_{\sigma}$ satisfies 
\begin{align*}
g_{\sigma}(\omega,\omega) \in B_{>  0}, & & \textrm{ for all non-zero } \omega \in \Omega^\bullet.
\end{align*}
As established in \cite[Lemma 5.2]{MMF3}, the Lefschetz decomposition is orthogonal with respect to $g_{\sigma}$, as is the $\mathbb{Z}^2_{\geq 0}$-decomposition of the complex structure.

\subsection{A Representation of $\frak{sl}_2$} \label{subsection:sl2}

As usual, we denote by $\frak{sl}_2$ the Lie algebra of trace-free $(2 \times 2)$-matrices, and take its standard basis 
\begin{align*}
e:= \begin{bmatrix}
0 & 1 \\
0 & 0
\end{bmatrix}\!, & & h := \begin{bmatrix}
1 & 0 \\
0 & -1
\end{bmatrix}\!, & & f:= \begin{bmatrix}
0 & 0 \\
1 & 0
\end{bmatrix}\!.
\end{align*}
As is readily verified \cite[Lemma 5.11]{MMF3}, the Lefschetz map $L_{\sigma}$ is adjointable on $\Omega^\bullet$ with respect to $g_{\sigma}$. Explicitly
\begin{align*} 
\Lambda_{\sigma} := L_{\sigma}^\dagger = \ast_{\s}^{-1} \circ L_{\sigma} \circ \ast_{\s}.
\end{align*}
Taking $L_{\sigma}$ and $\Lambda_{\sigma}$ together with the \emph{form degree counting operator}
\begin{align*}
H: \Omega^\bullet \to \Omega^\bullet, & & \omega \mapsto (k-n)\omega,  ~~ \text{ for } \omega \in \Omega^k,
\end{align*}
we get the following commutator relations:
\begin{align*}
[H,L_{\sigma}] = 2 L_{\sigma}, & & [L_{\sigma},\Lambda_{\sigma}] = H, & & [H,\Lambda_{\sigma}] = - 2  \Lambda_{\sigma}.
\end{align*}
Thus any K\"ahler structure gives a Lie algebra representation $T: \frak{sl}_2 \to  \frak{gl}\left(\Omega^\bullet\right)$ defined by $T(e) = E$, $T(h) = H$, and $T(f) = F$. Moreover, for any left $B$-module $\F$, we can extend this to a representation on $\Omega^{\bullet}  \otimes_B \F$ using the three operators
$L_{\F} := L_{\sigma} \otimes \id_{\F}$, $H_{\F} := H \otimes \id_{\F}$, and  $\Lambda_{\F} := \Lambda_{\sigma} \otimes \id_{\F}$.

\subsection{Holomorphic Modules}

Motivated by the Koszul--Malgrange characterisation of holomorphic bundles \cite{KoszulMalgrange}, noncommutative holomorphic modules have been considered  a number of times in the literature, see for example \cite{BS}, \cite{KLvSPodles}, \cite{OSV}, or \cite{PolishSch}.

For $\Omega^\bullet$ a dc over an algebra $B$, and $\mathcal{F}$ a left $B$-module, a \emph{(left) connection} for $\F$ is a $\mathbb{C}$-linear map $\nabla:\mathcal{F} \to  \Omega^1 \otimes_B \F $ satisfying 
\begin{align*}
\nabla(bf) =  \exd b \otimes f + b \nabla f, & & \textrm{ for all } b \in B, \, f \in \F.
\end{align*}
For a $B$-bimodule $\F$, a \emph{bimodule connection} for $\F$ is a pair $(\nabla,\sigma)$, where $\nabla$ is a connection and  $\sigma: \F \otimes_B \Omega^1 \to \Omega^1 \otimes_B \F$ is a bimodule map satisfying
\begin{align}\label{eqn:BiC}
\nabla(fb) = \nabla(f)b + \sigma (f \otimes \exd b), & & \textrm{ for all } f \in \F, \, b \in B.
\end{align}
We note that since $\sigma(f \otimes \exd b) = \nabla(fb)-\nabla(f)b$, the bimodule map $\sigma$ is uniquely determined by $\nabla$.

For a choice $\Omega^{(\bullet,\bullet)}$ of complex structure on $\Omega^{\bullet}$, a \emph{$(0,1)$-connection on $\mathcal{F}$} is a connection with respect to the dc $(\Omega^{(0,\bullet)},\adel)$. Any connection can be extended to a $\mathbb{C}$-linear map $\nabla: \Omega^{\bullet} \otimes_B \F  \to \Omega^{\bullet} \otimes_B \F$  by  defining 
\begin{align*}
\nabla(\omega \otimes f) =  \exd \omega \otimes f  +  (-1)^{|\omega|} \omega \wedge \nabla(f), & & \textrm{ for } \omega \in \Omega^{\bullet}, \, f \in \F,
\end{align*}
where $\omega$ is a homogeneous element, of the graded algebra $\Omega^{\bullet}$, of degree $|\omega|$. 
The \emph{curvature} of a connection is the left $B$-module map $\nabla^2: \mathcal{F} \to \Omega^2 \otimes_B \mathcal{F}$. A connection is said to be {\em flat} if $\nabla^2 = 0$. Note that a connection is flat if and only if  the pair $(\Omega^{\bullet} \otimes_B \F, \nabla)$ is a complex. 

\begin{defn}
For an algebra $B$, a \emph{holomorphic module over $B$} is a pair $(\mathcal{F},\adel_{\mathcal{F}})$, where  $\mathcal{F}$ is a finitely generated projective left $B$-module, and  $\adel_{\mathcal{F}}: \mathcal{F} \to  \Omega^{(0,1)}  \otimes_B \F$ is a flat $(0,1)$-connection, which we call the \emph{holomorphic structure} of $(\F, \adel_{\F})$. 
\end{defn}

\subsection{Hermitian Modules} \label{subsection:HVBS}

When $B$ is a $*$-algebra, we can also generalise the classical notion of an Hermitian metric for a complex vector bundle. For a $B$-bimodule $\F$, denote by $\F^{\vee}$ the space $\mathrm{Hom}_{B}(\F, B)$ of right $B$-module maps, which is a $B$-bimodule with respect to the left and right multiplications
\begin{align*}
(b\phi)(f) := b\phi(f), & & (\phi b)(f) := \phi(bf), & & \text{ for } b \in B, \, \phi \in \F^{\vee}, \textrm{ and } f \in \FF.
\end{align*}
Moreover, we denote by $\overline{\F}$ the \emph{conjugate module} of $\F$, which is a $B$-bimodule with respect to the left and right multiplications
\begin{align*}
b\overline{f} = \overline{fb^*}, & & \overline{f}b = \overline{b^*f} & & \textrm{ for } b \in B, \, f \in \F. 
\end{align*}
\begin{defn} \label{defn:Hermitianmetric} 
A \emph{weak Hermitian module} over a $*$-algebra $B$ is a pair $(\F,h_{\F})$, where $\F$ is a  $B$-bimodule and $h_{\F}:\overline{\F} \to \F^{\vee}  $ is a $B$-bimodule isomorphism, such that, for  the associated sesquilinear pairing, 
\begin{align*}
h'_{\F}(-,-): \F \times \F \to B, & & (f,k) \mapsto h_{\F}(\overline{f})(k)
\end{align*}
it holds  that, for all $f,k \in \F$, 
\begin{align*}
1. ~ h'_{\F}(f,k) = h'_{\F}(k,f)^*, & &  2. ~ h'_{\F}(f,f) \in B_{>0}, \textrm{ when } f \neq 0.
\end{align*}
\end{defn} 
Some immediate consequences of the definition are the following useful identities:
\begin{align} \label{eqn:hSUBFproperties}
h'_{\F}(bf,k) = h'_{\F}(f,b^*k), & &  h'_{\F}(fb,kc) = b^*h'_{\F}(f,b^*k)c, 
\end{align}
which hold for all $b,c \in B, \, f,k \in \F$. 

We can associate to any weak Hermitian module $(\F,h_{\F})$ the \emph{Hermitian metric}
\begin{align*}
g_{\F}: \Omega^{\bullet} \otimes_B \F  \times \Omega^{\bullet} \otimes_B \F   \to B.
\end{align*}
uniquely determined by 
$$
g_{\F}(\omega \otimes f, \nu \otimes k) := h'_{\F}\big(f, g_{\sigma}(\omega,\nu)k\big).
$$ 
As shown in  \cite[\textsection 5]{OSV}, the metric is conjugate symmetric, which is to say
\begin{align} \label{eqn:conjsymm}
g_{\F}(\alpha,\beta) = g_{\F}(\beta,\alpha)^*, & & \textrm{ for all } \alpha,\beta \in \Omega^{\bullet} \otimes_B \F .
\end{align}
Moreover, we note that the $\mathbb{Z}^2_{\geq 0}$-decomposition of $\Omega^{\bullet} \otimes_B \F $ is orthogonal with respect to $g_{\F}$.

\begin{defn}
If $g_{\F}$ gives $\Omega^{\bullet} \otimes_B \F$ the structure of a weak Hermitian module then we say that $(\F,h_{\F})$ is an \emph{Hermitian module}.
\end{defn}

\begin{remark}
In \cite{OSV} the pairings $g_{\sigma}$ and $g_{\F}$ were chosen to be linear in the first variable, whereas here we take linearity in the second variable. (This is required in order that the representation in Proposition \ref{prop:star.map} is a $*$-map.) However, as a careful reading of the \cite{OSV} will confirm, this does not affect any of the results used above.
\end{remark}

\subsection{Chern Connections and Positive Hermitian Holomorphic Modules}

An \emph{Hermitian holomorphic module} is a triple $(\mathcal{F},h_\F,\adel_{\mathcal{F}})$, where $(\mathcal{F},h_\F)$ is an Hermitian module, and $(\mathcal{F},\adel_{\mathcal{F}})$ is a holomorphic module. For any Hermitian holomorphic module $(\mathcal{F},h_\F,\adel_{\mathcal{F}})$, there is a direct noncommutative generalisation of the classical definition of an Hermitian connection, namely a connection $\nabla:\mathcal{F} \to  \Omega^{1} \otimes_B \FF  $ which is compatible with $h_{\F}$ in the sense of \cite{BeggsMajidChern}, see also the accompanying paper \cite[\textsection 2.9]{SpectralGap}. For any  Hermitian holomorphic module $(\mathcal{F},h_\F,\adel_{\mathcal{F}})$, there exists a unique Hermitian connection $\nabla:\mathcal{F} \to  \Omega^1 \otimes_B \F$ satisfying
\begin{align*}
 \adel_{\F} = \big(\id \otimes \mathrm{proj}_{\Omega^{(0,1)}}\big) \circ  \nabla,
\end{align*}
where $\mathrm{proj}_{\Omega^{(0,1)}}$ is the obvious projection from $\Omega^1$ to $\Omega^{(0,1)}$. Moreover, defining $\mathrm{proj}_{\Omega^{(1,0)}}$ in direct analogy with $\mathrm{proj}_{\Omega^{(0,1)}}$, we denote
$
\del_{\F} := \big(\mathrm{proj}_{\Omega^{(1,0)}} \otimes \id\big) \circ  \nabla,
$
and call $\nabla$ the \emph{Chern connection} of Hermitian holomorphic module $(\mathcal{F},h_\F, \adel_{\mathcal{F}})$. This allows us to introduce a notion of positivity for a holomorphic Hermitian module. This directly generalises the classical notion of positivity, a property which is equivalent to ampleness \cite[Proposition 5.3.1]{HUY}. It was first introduced in \cite[Definition 8.2]{OSV} and requires a compatibility between Hermitian holomorphic modules and K\"ahler structures.

\begin{defn}\label{defn:positiveNegativeVB}
Let $\Omega^{\bullet}$ be a dc over a $*$-algebra $B$, and let $(\Omega^{(\bullet,\bullet)},\kappa)$ be a K\"ahler structure for $\Omega^{\bullet}$. 
An Hermitian  holomorphic module $(\mathcal{F},h_{\F}, \adel_{\mathcal{F}})$  is said to be \emph{positive}, written $\F > 0$, if there exists $\theta \in \mathbb{R}_{>0}$ such that the Chern connection $\nabla$ of ${\mathcal{F}}$ satisfies
\begin{align*}
\nabla^2(f) = -\theta \mathbf{i} L_{\F}(f) = -\theta \mathbf{i} f \otimes \kappa, & & \textrm{ for all } f \in \mathcal{F}.
\end{align*} 
Analogously, $(\mathcal{F}, h_{\F}, \adel_{\mathcal{F}})$ is said to be {\em negative}, written $\F <0$, if there exists $\theta \in \mathbb{R}_{>0}$ such that the Chern connection $\nabla$ of ${\mathcal{F}}$ satisfies
\begin{align*}
\nabla^2(f) = \theta \mathbf{i} L_{\F}(f) =  \theta \mathbf{i} f \otimes \kappa, & & \textrm{ for all } f \in \mathcal{F}.
\end{align*} 
\end{defn}
As established in \cite[\textsection 8]{OSV}, for positive Hermitian holomorphic modules we have a direct noncommutative generalisation of the Kodaira vanishing theorem. Moreover, as shown in the accompanying paper \cite[\textsection 3]{OSV}, twisting a noncommutative Dolbeault--Dirac operator by a negative Hermitian holomorphic module produces a spectral gap around zero.

\subsection{Inner Products, and Twisted Dolbeault--Dirac and Laplace Operators}

Let $B$ be a $*$-algebra, then a \emph{state} $\mathbf{f}:B \to \mathbb{C}$ is a $*$-map satisfying $\mathbf{f}(b^*b) \geq 0$, for all $b \in B$, and such that $f(1_B) = 1$. If in addition $f(b^*b) > 0$, for all non-zero $b$, then we say that $\mathbf{f}$ is \emph{faithful}. As observed in \cite[\textsection 5.2]{OSV}, an inner product is then given by 
\begin{align*}
\langle -, - \rangle_{\F} := \mathbf{f} \circ g_{\F}:  \Omega^{\bullet} \otimes_B \F  \times  \Omega^{\bullet} \otimes_B \F  \to \mathbb{C}.
\end{align*}
Moreover, we can associate to $\mathbf{f}$ an \emph{integral}
$
\int := \mathbf{f} \circ \ast_{\sigma}: \Omega^{2n} \to \mathbb{C}.
$
We say that $\int$ is \emph{closed} if the map
$
\int \circ \, \exd:\Omega^{n-1} \to \mathbb{C}
$
is equal to the zero map.

As shown in \cite[Proposition 5.15]{OSV}, for any Hermitian holomorphic module $(\F,\adel_{\F})$ over a K\"ahler structure $(B,\Omega^{\bullet},\Omega^{(\bullet,\bullet)},\sigma)$, with closed integral $\int$ with respect to a choice of faithful state, the twisted differentials $\del_{\F}$ and $\adel_{\F}$ are  adjointable with respect to $\langle -,-\rangle_{\F}$, with adjoints denoted by $\del_{\F}^{\dagger}$ and $\adel_{\F}^{\dagger}$ respectively. The {\em $\F$-twisted holomorphic} and \emph{anti-holomorphic Dolbeault--Dirac operators} are respectively defined to be
\begin{align*}
D_{\del_\F} := \del_\F + \del_\F^\dagger, & & D_{\adel_\F} :=  \adel_\F + \adel_\F^\dagger.
\end{align*}
Its  {\em $\F$-twisted  holomorphic} and \emph{anti-holomorphic Laplace operators} are respectively 
\begin{align*}
\DEL_{\del_\F} \coloneqq D_{\del_{\F}}^2, & & \DEL_{\adel_\F} \coloneqq D_{\adel_{\F}}^2.
\end{align*}
Adjointability of the operators $\del_{\F}$ and $\adel_{\F}$ immediately implies that $\nabla = \del_{\F} + \adel_{\F}$ admits an adjoint $\nabla^{\dagger}$, giving us respectively the  \emph{Chern--Dirac}, and \emph{Chern--Laplace, operators}
\begin{align*}
D_{\nabla} := \nabla + \nabla^{\dagger}, & & \Delta_{\nabla} := \nabla \circ \nabla^{\dagger} + \nabla^{\dagger} \circ \nabla.
\end{align*}
Using the Laplacians one can define \emph{harmonic elements} just as in the classical case: 
\begin{align*}
\mathcal{H}_{\nabla}^{\bullet} := \mathrm{ker}(\Delta_{\nabla}), & & \mathcal{H}_{\del_{\F}}^{\bullet} := \mathrm{ker}(\Delta_{\del_{\F}}), & & \mathcal{H}_{\adel_{\F}}^{\bullet} := \mathrm{ker}(\Delta_{\adel_{\F}}).
\end{align*}
As established in \cite[\textsection 6.1]{OSV}, if the twisted Dirac operator is diagonalisable, then we have a direct noncommutative generalisation of Hodge decomposition, giving a bijection between cohomology classes and harmonic forms. Moreover, in this case, classical Serre duality also carries over to the noncommutatve setting \cite[\textsection 6.2]{OSV}.

\subsection{Covariant Differential Calculi and Hermitian Structures} \label{subsection:covariantHS}

We begin by briefly recalling Takeuchi's equivalence for relative Hopf modules, see \cite[Appendix A]{SpectralGap} for more details. For $A$ a Hopf algebra, with coproduct $\Delta$, counit $\e$, and an invertible antipode $S$. We say that a left coideal subalgebra $B \subseteq A$ is a \emph{quantum homogeneous \mbox{$A$-space}} if $A$ is faithfully flat as a right $B$-module and  $B^+A = AB^+$, where $B^+ := \ker(\e) \cap B$. We denote by ${}^A_B\mathrm{mod}$ the category of (left) relative Hopf modules which are finitely generated as left $B$-modules, and by ${}^{\pi_B}\mathrm{mod}$ the category of finite-dimensional left comodules over the Hopf algebra $\pi_B(A) := A/B^+A$. An equivalence of categories, called Takeuchi's equivalence, is given by the functor $\Phi:{}^A_B\mathrm{mod} \to {}^{\pi_B}\mathrm{mod}$, where $\Phi(\F) = \F/B^+\F$, for any relative Hopf module $\F$, and the functor $\Psi:{}^{\pi_B}\mathrm{mod} \to {}^A_B\mathrm{mod}$, defined by $\Psi(V) = A \square_{\pi_B} V$, where $\square_{\pi_B}$ is the cotensor product over $\pi_B(A)$. A unit for the equivalence is given by $\unit: \F \to (\Psi \circ \Phi)(\F)$, where $\unit(f) := f_{(-1)} \otimes [f_{(0)}]$, with $[f_{(0)}]$ denoting the coset of $f_{(0)}$ in $\Phi(\F)$. 

Consider next the category ${}^A_B\mathrm{mod}_0$ of \emph{relative Hopf modules}, consisting of objects $\F$ in ${}^A_B\textrm{mod}$ endowed with the right $B$-module structure defined by $fb = f_{(-2)}bS(f_{(-1)})f_{(0)}$, for $b \in B$, and $f \in \F$. We note that this immediately implies the identity 
\begin{align} \label{eqn.ModZeroRight}
bf = f_{(0)}S^{-1}(f_{(-1)})bf_{(-2)}, & & \textrm{ for all } b \in B, \, f \in \F.
\end{align}
Moreover, we note that
\begin{align} \label{eqn:cosetModue}
[bf] = [fb] = \e(b), & & \textrm{ for } \F \in {}^A_B\mathrm{mod}_0, \, b \in B.
\end{align}

This category is clearly equivalent to ${}^A_B\textrm{mod}$. The category ${}^A_B\mathrm{mod}_0$ comes equipped with an evident monoidal structure $\otimes_B$, moreover, the category ${}^{\pi_B}\mathrm{mod}$ comes with the usual tensor product of comodules. For $\F,\mathcal{G} \in {}^A_B\mathrm{mod}_0$, the natural isomorphism 
\begin{align*}
\mu_{\F,\mathcal{G}}:\Phi(\F) \otimes \Phi(\mathcal{G}) \to  \Phi(\F \otimes_B \mathcal{G}), & & [f] \otimes [g] \mapsto [f \otimes g].
\end{align*}
makes Takeuchi's equivalence into a monoidal equivalence. See \cite[\textsection 4]{MMF2} for further details.

For any $\F \in {}^A_B\mathrm{mod}_0$, the $B$-bimdoule $\F^{\vee}$ is a right dual to $\F$ in the monoidal category of $B$-bimodules. Note next that in the category ${}^{\pi_B}\mathrm{mod}$, the dual comodule $V^{\vee} \in {}^{\pi_B}\mathrm{mod}$ is a right dual to $V$. Hence, from the monoidal version of Takeuchi's equivalence, we see that $\Psi(\Phi(\F)^{\vee})$ is also a right dual to $\F$. It now follows from uniqueness of duals in the monoidal category of $B$-bimodules that  $\F^{\vee}$ is isomorphic to $\Psi(\Phi(\F)^{\vee})$, endowing it with the structure of an object in ${}^{A}_B\mathrm{mod}_0$. In what follows, we will tacitly identity these two objects.

Let us now assume that $B$ is a $*$-algebra. An \emph{covariant Hermitian module} over $B$ is an Hermitian module $(\F, h_\F)$ such that $\F$ is an object in ${}^A_B\mathrm{mod}_0$ and the isomorphism 
$
h_{\F}: \overline{\F} \to \, \F^{\vee}
$
is a morphism in ${}^A_B\textrm{mod}_0$, where the conjugate module $\overline{\F}$ is understood as an object in ${}^A_B\textrm{mod}_0$ in the obvious way. 

We say that a dc $\Omega^{\bullet}$ over $B$ is \emph{covariant} if the coaction $\Delta_R:B \to A \otimes B$ extends to a (necessarily unique) comodule algebra structure $\Delta_R: \Omega^{\bullet} \to A \otimes \Omega^{\bullet}$, with respect to which the diﬀerential $\exd$ is a left $A$-comodule map. Note that any covariant dc, which is finitely generated as a left $B$-module, is naturally an object in the category ${}^A_B\mathrm{mod}$. In this article, we will \emph{always} assume that the left $B$-module structure of $\Omega^{\bullet}$ gives it the structure of an object in the category ${}^A_B\mathrm{mod}_0$. A connection $\nabla:\F \to \Omega^1 \otimes_B \F$ is said to be \emph{covariant} if it is a left $A$-comodule map.

Consider next a covariant dc $\Omega^{\bullet}$ over $B$, endowed with a \emph{covariant} complex structure $\Omega^{(\bullet,\bullet)}$, which is to say, one for which the $\mathbb{Z}^2_{\geq 0}$-decomposition of the calculus is a decomposition in the category of relative Hopf modules. A \emph{covariant holomorphic module} is a holomorphic module $(\F,\adel_{\F})$ over $B$, such that $\F$ is an object in ${}^A_B\textrm{mod}_0$ and $\adel_{\F}:\F \to  \Omega^{(0,1)}  \otimes_B \F$ is a covariant connection. An Hermitian holomorphic module $(\F, h_\F, \del_F)$ is said to be \emph{covariant} if its constituent Hermitian and holomorphic modules are covariant. In this case, the Chern connection is always covariant, see \cite[\textsection 7.1]{OSV}.

A {\em covariant Hermitian} structure for $\Om^\bullet$  is an Hermitian structure $(\Om^{(\bullet,\bullet)},\s)$ such that $\Om^{(\bullet,\bullet)}$ is a covariant complex structure, and the Hermitian form $\sigma$ is left $A$-coinvariant, which is to say $\DEL_R(\s) = 1 \otimes \s$. A {\em covariant K\"ahler} structure is a covariant Hermitian structure that is also a K\"ahler structure.  Note that in the covariant case, in addition to being $B$-bimodule maps, $L_{\sigma}$, $\ast_\s$, and $\Lambda_{\sigma}$ are also left $A$-comodule maps. 

\begin{eg}
Let $(\Omega^{(\bullet,\bullet)}, \s)$ be a positive definite covariant Hermitian structure for a $*$-dc over a quantum homogeneous space $B$ that is also a $*$-algebra. A covariant Hermitian module structure for 
$\Omega^{\bullet}$ is given by the map
\begin{align*}
\overline{\Omega^{\bullet}} \to (\Omega^{\bullet})^{\vee}, & & \overline{\omega} \mapsto \ast_{\sigma}(\ast_{\sigma}(\omega^*) \wedge - ),
\end{align*}
where $\ast_{\sigma}(\ast_{\sigma}(\omega^*) \wedge - )$ acts on an element of $\Omega^{\bullet}$ in the obvious way.
\end{eg}


\section{Compact Quantum Homogeneous K\"ahler Spaces}

In this section we introduce the main object of study in this paper, mainly the notion of a compact quantum homogeneous K\"ahler space, and more generally, the Hermitian analogue. This gives a natural set of compatibility conditions between covariant Hermitian structures and Woronowicz's theory of  compact quantum groups. It has a natural Hilbert space completion together with a rich family of geometrically motivated bounded operators. We refer to Appendix \ref{app:CQGs} for basic definitions and notation for compact quantum groups (CQGs) and compact quantum group algebras (CQGAs).


\subsection{Square Integrable Forms} \label{subsection:HilbertL2Omega}

In this subsection we introduce the Hilbert space of square integrable forms for a positive definite Hermitian structure $(\Omega^{(\bullet,\bullet)},\sigma)$ over a $*$-algebra $B$, twisted by an Hermitian module $\F$. For a faithful state $\textbf{f}: B \to \mathbb{C}$, an inner product  $\langle -,-\rangle_{\F}$ is given by $\textbf{f} \circ g_{\F}$. We denote by $L^2\left(\Om^{\bullet} \otimes_B \F \right)$ the Hilbert space completion of $\Omega^{\bullet} \otimes_B \F$ with respect to $\langle -, -\rangle_{\F}$, and call it the {\em Hilbert space of square integrable twisted forms of } $\Omega^{\bullet} \otimes_B \F$.

We note that the following Hilbert space decompositions  
\begin{align*} 
L^2(\Omega^{\bullet} \otimes_B \F) \cong \bigoplus_{(a,b) \in \mathbb{Z}^2_{\geq 0}} L^2(\Om^{(a,b)} \otimes_B  \F), & &  L^2(\Om^{\bullet} \otimes_B  \F) \cong \bigoplus_{j \in \mathbb{Z}_{\geq 0}}  L^2\big(L_{\sigma}^j(P^{(2n-2j} \otimes_B \F)\big)
\end{align*}
follow from  the fact that the $\mathbb{Z}^2_{\geq 0}$-decomposition, and the Lefschetz decomposition, of $\Omega^\bullet \otimes_B \F$ are orthogonal with respect to $g_{\sigma}$ (as discussed in \textsection \ref{subsection:HandKS})), and hence are orthogonal with respect to the associated inner product. Moreover, if the twisted Dirac operator $D_{\adel_{\F}}$ is diagonalisable, then the additional decomposition
\begin{align*}
L^2(\Omega^{\bullet} \otimes_B \F) \cong L^2\!\Big(\adel_{\F}(\Omega^{\bullet} \otimes_B \F)\Big) \oplus L^2\!\left(\adel^{\dagger}_{\F}(\Omega^{\bullet} \otimes_B \F)\right) \oplus L^2(\mathcal{H}_{\adel_{\F}})
\end{align*}
follows from the orthogonality of Hodge decomposition \cite[Theorem 6.4]{OSV}.

\subsection{CQH-Hermitian Spaces}

We now introduce the class of quantum homogeneous spaces that concern us in this paper. Note first that for any quantum homogeneous space $B \subseteq A$, where $A$ is a Hopf $*$-algebra, and $B$ is a $*$-subalgebra of $A$, the quotient Hopf algebra $A/B^+A$ inherits a Hopf $*$-algebra structure.  A {\em CQGA-homogeneous space} is a quantum homogeneous space $B \subseteq A$ such that $A$ and $\pi_B(A)$ are CQGAs, $B$ is a $*$-subalgebra of $A$, and $\pi_B$ is a Hopf $*$-algebra map. It follows that $A$ is faithfully flat as a right $B$-module, as explained, for example, in \cite[\textsection 3.3]{DOKSS}.

\begin{defn}
A \emph{compact quantum  homogeneous Hermitian space}, or simply a \emph{CQH-Hermitian space}, is a quadruple $\mathbf{H} := \big(B \subseteq A, \Omega^\bullet, \Omega^{(\bullet,\bullet)},\sigma \big)$  where
\begin{enumerate}

\item $B \subseteq A$  is a CQGA-homogeneous space,

\item $\Omega^\bullet$ is a left $A$-covariant $*$-dc over $B$, 

\item $\big(\Omega^{(\bullet,\bullet)},\sigma\big)$ is a left $A$-covariant, $\int$-closed, positive definite Hermitian structure for  \mbox{ $\Omega^\bullet$}.

\end{enumerate}
A \emph{compact quantum  homogeneous K\"ahler space}, or simply a  \emph{CQH-K\"ahler space} is a CQH-Hermitian space whose constituent Hermitian holomorphic structure is a K\"ahler structure.
\end{defn}

Since the Haar state $\haar$ of $A$ is a faithful state, we can follow the approach of \textsection \ref{subsection:HilbertL2Omega} and form inner products from twisted metrics, and moreover consider the associated Hilbert space completions.  We now give a sufficient condition for separability of such Hilbert spaces.

\begin{prop} \label{prop:separable}
For a quantum homogeneous space $\mathbf{H} = (B \subseteq A, \Omega^{\bullet}, \Omega^{(\bullet,\bullet)},\sigma)$, and a covariant Hermitian module $\F$, if $A$ is finitely generated, then the Hilbert space $L^2(\Omega^{\bullet} \otimes_B \F  )$ is separable.
\end{prop}
\begin{proof}
Since $A$ is finitely generated, it admits a countable Hamel basis. Thus the tensor product $A \otimes \Phi(\Omega^{\bullet} \otimes_B \F)$, as well as the subspace $A \square_{\pi_B} \Phi(\Omega^{\bullet} \otimes_B \F )$, admit such a basis. Since $\Omega^{\bullet} \otimes_B \F $ is isomorphic to the cotensor product $A \square_{\pi_B} \Phi(\Omega^{\bullet} \otimes_B \F)$, it also admits a countable basis. Thus we can conclude that $L^2(\Omega^{\bullet} \otimes_B \F )$ is separable.
\end{proof}

We finish this subsection with the observation that in the covariant case, weak Hermitian  modules automatically satisfy the stronger requirement of an Hermitian module.

\begin{prop}
Let $\F \in {}^A_B\mathrm{mod}_0$, and let $h_{\F}: \overline{\F} \to \F^{\vee}$ be a covariant weak Hermitian structure for $\F$. Then $h_{\F}$ is automatically an Hermitian structure.
\end{prop}
\begin{proof}
Since $h_{\F}$ is a weak Hermitian structure by assumption, it follows from \cite[Proposition 4.3]{SpectralGap} that an inner product is given by 
\begin{align*}
\langle -,-\rangle_{\e,\F}: \Phi(\F) \times \Phi(\F) \to \mathbb{C}, & & ([f],[k]) \mapsto \e(h'_{\F}(f,k)).
\end{align*}
In particular, there exist objects $f_i$ in $\F$ such that $[f_i]$ is an orthonormal basis of $\Phi(\F)$ with respect to $\langle -,-\rangle_{\e,\F}$. Analogously, we have objects $\omega_i$ in $\Omega^{\bullet}$ such that $[\omega_i]$ is an orthonormal basis of $\Phi(\Omega^{\bullet})$ with respect to the inner product $\langle -,-\rangle_{\e,g_{\sigma}})$. Then consider a general element $\sum_i [\omega_i] \otimes [f_i]$ in $\Phi(\Omega^{\bullet}) \otimes \Phi(\F)$, and take the product
\begin{align*}
\e\Big(g_{\F}(\sum_i \omega_i \otimes f_i, \sum_i \omega_i \otimes f_i)\Big) = \e(h_{\F}'(f_i, g_{\sigma}(\omega_i,\omega_j)f_j).
\end{align*}
Recalling from \eqref{eqn.ModZeroRight} the form of the left $B$-module structure of $\F$ as an object in ${}^A_B\mathrm{mod}_0$, as well as  right linearity of $h_{\F}$, we see that this is in turn equal to s
\begin{align*}
\e\Big(h_{\F}'\big(f_i,(f_j)_{(0)}\big)\Big) \e\Big(S^{-1}\big((f_j)_{(-1)}\big)g_{\sigma}(\omega_i,\omega_j)^*(f_j)_{(-2)}\Big).
\end{align*}
Simplifying we arrive at the expression
\begin{align*}
\e\Big(h_{\F}'(f_i,f_j)\Big) \e\Big(g_{\sigma}(\omega_i,\omega_j)\Big)^* = 1,
\end{align*}
where for the last equality we have used our assumption that the sets $\{[\omega_i]\}_i$ and $\{[f_j]\}_j$ are orthonormal bases for their respective inner product spaces. It now follows that $\e \circ g_{\F}$ is an inner product, and so, from  \cite[Proposition 4.3]{SpectralGap} we see that $g_{\F}$ is an Hermitian structure, meaning in turn that $h_{\F}$ is Hermitian.
\end{proof}


\subsection{Morphisms as Bounded Operators} \label{subsection:Boundedsl2}

In this subsection, for a CQH-Hermitian space $\mathbf{H} := \big(B \subseteq A, \Omega^\bullet, \Omega^{(\bullet,\bullet)},\sigma \big)$, and a covariant Hermitiam module $\F$, we discuss the extension of endomorphisms of $\Omega^{\bullet} \otimes_B \F $ to bounded operators on $L^2(\Omega^{\bullet} \otimes_B \F)$. As an application, we produce a bounded representation of $\frak{sl}_2$.

Consider the map
$
g_{\F,\unit}:  A \, \square_{\pi_B} \Phi(\Omega^{\bullet} \otimes_B \F )  \times  A \, \square_{\pi_B} \Phi(\Omega^{\bullet} \otimes_B \F)  \to B,
$
defined by 
$$
g_{\F,\unit}\!\left(\sum_i a_i \otimes [\alpha_i], \sum_j a'_j \otimes [\alpha'_j]\right)  := \sum_{i,j} a^*_ia'_j \e(g_{\F}(\alpha_i,\alpha'_j)). 
$$

The composition $\mathbf{h} \circ g_{\F,\unit}$ is an inner product for  $A \, \square_{\pi_B}\Phi(\Omega^{\bullet} \otimes_B \F)$. Moreover, it follows from \cite[Lemma 4.1]{SpectralGap} that the unit map
$
\unit: \Omega^{\bullet} \otimes_B \F   \to  A \, \square_{\pi_B}\Phi(\Omega^{\bullet} \otimes_B \F)
$
is an isometry.

\begin{prop} \label{prop:boundedmorph}
Every morphism $f:\Omega^{\bullet} \otimes_B \F   \to \Omega^{\bullet} \otimes_B \F$ in ${}^A_B\mathrm{mod}_0$ is bounded, and hence extends to a bounded operator on $L^2(\Omega^{\bullet} \otimes_B \F )$. 
\end{prop}
\begin{proof}
Consider the commutative diagram given by Takeuchi's equivalence
\begin{align*}
\xymatrix{ 
\Omega^{\bullet} \otimes_B \F           \ar[d]_{\unit}           \ar[rrr]^{f}       & & &     \Omega^{\bullet} \otimes_B \F    \\
A \square_{\pi_B} \big(\Phi(\Omega^\bullet) \otimes \Phi(\F)\big)  \,\,   \ar[rrr]_{\Psi \circ \Phi(f)}                      & & &  A \square_{\pi_B} \Phi(\Omega^\bullet).   \ar[u]_{\unit^{-1}} \\
}
\end{align*}
Since $\unit$ is an isometry, the morphism $f$ is bounded if and only if $\Psi \circ \Phi(f)$ is bounded. But $\Psi \circ \Phi(f) = \Phi(f) \otimes \id$, and $\Phi(\Omega^{\bullet} \otimes_B \F)$ is finite-dimensional by assumption, implying that $\Phi(f) \otimes \id$ is bounded. Hence $f$ is bounded and extends to a bounded operator on $L^2(\Omega^{\bullet} \otimes_B \F)$.
\end{proof}

\begin{cor} \label{cor:sl2rep}
The maps $\id_{\F} \otimes L_{\sigma},\, \id_{\F} \otimes \Lambda_{\sigma}$, and $\id_{\F} \otimes H$ extend to bounded operators on $L^2(\Omega^{\bullet} \otimes_B \F)$. Hence, a representation $T:\frak{sl}_2 \to \mathbb{B}\!\left(L^2(\Omega^{\bullet} \otimes_B \F)\right)$ is given by 
\begin{align*}
T(e) = L_{\sigma} \otimes \id_{\F}, & &
T(h) = H \otimes \id_{\F}, & & 
T(f) = \Lambda_{\sigma} \otimes \id_{\F}.
\end{align*}
The space of lowest weight vectors of the representation is given by $L^2(P^\bullet \otimes_B \F)$, the Hilbert space completion of the primitive forms. 
\end{cor}
\begin{proof}
Since $L_{\sigma} \otimes  \id_{\F}, ~ \Lambda_{\sigma} \otimes  \id_{\F}$, and $H \otimes  \id_{\F}$ are all morphisms in ${}^A_B\mathrm{mod}_0$, Proposition \ref{prop:boundedmorph} implies that they extend to bounded operators on $L^2(\Omega^{\bullet} \otimes_B \F )$. It now follows from the $\frak{sl}_2$-representation given in \cite[Corollary 5.14]{MMF3} that we get a bounded Lie algebra representation of $\frak{sl}_2$.
\end{proof}

\begin{cor}
The Hodge map $\ast_{\sigma} \otimes  \mathrm{id}$ extends to a unitary operator on $L^2(\Omega^{\bullet} \otimes_B \F)$.
\end{cor}
\begin{proof}
Since $\ast_{\sigma}$ is a morphism in ${}^A_B\textrm{mod}_0$, it follows from Proposition \ref{prop:boundedmorph} that it is bounded on $\Omega^{\bullet}$. Unitarity of $\ast_{\sigma}$ as an operator on $\Omega^{\bullet}$ now follows from  \cite[Lemma 5.10]{MMF3}. Thus it follows from the definition of $g_{\F}$ that $\mathrm{id} \otimes \ast_{\sigma}$ acts as a unitary operator on $\Omega^{\bullet} \otimes_B \F$, and hence extends to a bounded operator on $L^2(\Omega^{\bullet} \otimes_B \F)$.
\end{proof}

\begin{cor}
The curvature operator $\nabla^2:\Omega^{\bullet} \otimes_B \F  \to \Omega^{\bullet} \otimes_B \F $ extends to a bounded operator on $L^2(\Omega^{\bullet} \otimes_B \F)$.
\end{cor}
\begin{proof}
Since $\nabla^2$ is a left $B$-module, and hence a morphism in ${}^A_B\textrm{mod}_0$, it follows from Proposition \ref{prop:boundedmorph} that it is bounded on $\Omega^{\bullet} \otimes_B \F $, and hence extends to a bounded operator on the Hilbert space.
\end{proof}

\begin{cor}
For a covariant bimodule connection $(\nabla,\sigma)$, the associated bimodule map $\sigma: \F \otimes_B \Omega^1 \rightarrow \Omega^{1} \otimes_B \F$ is a bounded operator.
\end{cor}
\begin{proof}
Since $\sigma(f \otimes \exd b) = \nabla(fb)-\nabla(f)b$, and $\nabla$ is by assumption a comodule map, we see that $\sigma$ must be a left $A$-comodule map. Thus since it is also a $B$-bimodule map, it must be a morphism in ${}^A_B\mathrm{mod}_0$. It now follows from Proposition \ref{prop:boundedmorph} that $\sigma$ is bounded.
\end{proof}


\subsection{Bounded Multiplication Maps} \label{subsection:BoundMultMaps}

Let $\mathbf{H} = (B,\Omega^{\bullet},\Omega^{(\bullet,\bullet)},\sigma)$ be a CQH-Hermitian space, and $\FF$ a covariant Hermitian holomorphic module. In this subsection we produce bounded operators on $\Omega^{\bullet}  \otimes_B \F$ using left multiplication by forms.

\begin{prop}\label{prop:formmult}
Assume that $\FF$ is a covariant Hermitian module, then for any element $\omega \in \Om^{\bullet}$, a bounded linear operator is given by 
\begin{align*}
L_\omega: \Omega^{\bullet}  \otimes_B \F  \to \Omega^{\bullet} \otimes_B \F,  & &  f \otimes \nu \mapsto  \omega \wedge \nu \otimes f.
\end{align*}
Hence it admits an extension $L_{\omega}$ to a bounded linear operator on $L^2(\Omega^{\bullet} \otimes_B \F)$.
\end{prop}
\begin{proof}
For any $\omega \in \Omega^1$, \, $b \in B^+$, and $\nu \otimes f \in \Omega^{\bullet} \otimes_B \F$, it follows from \eqref{eqn:cosetModue} that
\begin{align*}
[\omega \wedge b\nu] \otimes [f] = [\omega b \wedge \nu] \otimes [f] = \e(b)[\omega \wedge \nu] \otimes [f]  =  0. 
\end{align*}
Thus a well-defined linear map is given by 
\begin{align*}
l_{\omega}: \Phi(\Omega^{\bullet}) \otimes \Phi(\F) \to \Phi(\Omega^{\bullet})  \otimes \Phi(\F), & & [\nu] \otimes [f] \mapsto [\omega \wedge \nu] \otimes [f].
\end{align*}
Since $\Phi(\Omega^{1}) \otimes \Phi(\F)$ is finite-dimensional, $l_{\omega}$ is a bounded operator. Thus, recalling the multiplicative representation $\lambda_A:A \to \mathbb{B}(L^2(A))$ introduced in Appendix \ref{app:CQGs}, we see that a bounded operator on  $\Phi(\Omega^{\bullet} \otimes_B \F) \otimes A$ is given by 
$
\lambda_A(a) \otimes l_{\omega},  \text{ for all } a \in A, \,\nu \in \Omega^\bullet.
$


With this result in hand, we now show that $L_{\omega}$ is a bounded linear operator. For any element $\nu \otimes f$ of $\Omega^{\bullet}  \otimes_B \F$, we see that 
\begin{align*}
\unit \circ L_{\omega} \circ \unit^{-1}\left(\nu_{(-1)}f_{(-1)} \otimes [\nu_{(0)}] \otimes [f_{(0)}]\right) 
\end{align*}
simplifies to the expression 
\begin{align*}
\unit \circ L_{\omega}(\nu \otimes f)  = & \, \unit(\omega \wedge \nu \otimes f)\\
= & \, \omega_{(-1)}\nu_{(-1)}f_{(-1)} \otimes [\omega_{(0)} \wedge \nu_{(0)}] \otimes [f_{(0)}]\\
= & \, ((\lambda_{A}(\omega_{(-1)}) \otimes l_{\omega_{(0)}})(\nu_{(-1)}f_{(-1)} \otimes [\nu_{(0)}] \otimes [f_{(0)}]).
\end{align*}
Since every element of $A \square_{\pi_B} \Phi(\Omega^\bullet) \otimes \Phi(\F)$ is a sum of elements of the form $\unit(\nu \otimes f)$, for some $\nu \otimes f \in \Omega^1 \otimes_B \F$, we see that $\unit \circ L_{\w} \circ \unit^{-1}$ is bounded. It now follows from  the fact that $\unit$ is an isometry, that $L_{\omega}$ is a bounded operator. 


Finally, we observe that since $\Omega^{\bullet}  \otimes_B \F$ is dense in $L^2(\Omega^{\bullet} \otimes_B \F)$ by construction, $L_{\w}$ uniquely extends to an element of $\mathbb{B}(L^2(\Omega^{\bullet} \otimes_B \F))$.
\end{proof}

\begin{prop} \label{prop:star.map}
A $*$-algebra representation is given by 
\begin{align*} 
\lambda: B \to \mathbb{B}(L^2(\Omega^{\bullet}  \otimes_B \F)), & & b \mapsto \lambda(b) := L_b.
\end{align*} 
Moreover, if $A$ has no non-trivial zero divisors, or if $\FF = B$, then $\lambda$ is a faithful representation.
\end{prop}
\begin{proof}
It is clear that $\lambda$ is an algebra map. To see that $\lambda$ is a $*$-map, take $b \in B$, $\omega,\nu \in \Omega^{\bullet}$, and $f,k \in \F$, and note that 
\begin{align*}
g_{\F}(b \omega \otimes f, \nu \otimes \k)_{\sigma} 
= \, &  h'_{\F}(f, g_{\sigma}(b\omega,\nu)k)\\ 
= \, &  h'_{\F}(f, g_{\sigma}(\omega,b^*\nu)k)\\ 
= \, & g_{\F}(\omega \otimes f, b^*\nu \otimes \k)_{\sigma}\\ 
= \, & g_{\F}(\omega \otimes f, \lambda(b^*)\nu \otimes \k)_{\sigma}.
\end{align*}
Thus the adjoint map $\lambda^{\dagger}(b)$ coincides with $\lambda(b^*)$ on the dense subalgebra $\Omega^{\bullet} \otimes_B \F$, meaning that $\lambda^{\dagger}(b) = \lambda(b^*)$. Thus we see that $\lambda$ is a $*$-map as claimed.

Let us now assume that $A$ has no non-trivial zero divisors. Then it holds that $\FF \simeq A \square_B \Phi(\FF)$ is a torsion-free module, and hence that $\lambda$ is faithful. If instead we assume that $\FF = B$, then we see that since $1$ is an element of $\Omega^{\bullet}  \otimes_B \F\simeq \Omega^{\bullet}$, the representation must also be faithful.
\end{proof}

\subsection{Bounded Commutators} \label{subsection:BoundCommutators}

We now consider a second consequence of Proposition \ref{prop:formmult}, namely boundedness of the various commutator operators associated to a CQH-Hermitian space and an Hermitian holomorphic module. This is a direct noncommutative generalisation of an important classical phenomenon \cite[\textsection 2.4.1]{RennieSpecTrip}, one which is generalised by the definition of $K$-homology, and ultimately spectral triples, as recalled in Appendix \ref{app:spectraltriple}.

\begin{lem}\label{lem:bcomm}
For any $b \in B$, and $\omega \otimes f \in \Omega^\bullet \otimes_B \F$, it holds that 
\begin{enumerate}
    \item $[\nabla, \lambda(b)](\omega \otimes f) = \exd b \wedge \omega \otimes f$,
    \item $[\del_{\F}, \lambda(b)](\omega \otimes f) =\del b \wedge \omega \otimes f$,
    \item $[\adel_{\F}, \lambda(b)](\omega \otimes f) = \sigma(\adel b \otimes f) \wedge \omega$.
\end{enumerate}
\end{lem}
\begin{proof}
The first identity follows from the calculation 
\begin{align*}
[\nabla,\lambda(b)](\omega \otimes f) 
 = & \, \nabla(b\omega \otimes f) -  b \nabla(\omega \otimes f) \\
 = & \, \sigma(\exd b \otimes f) \wedge \omega + b\nabla(f) \wedge \omega - bf \otimes \exd \omega
   - b \nabla(f) \wedge \omega - bf \otimes \exd \omega\\
 = & \, \sigma(\exd b \otimes f) \wedge \omega.
\end{align*}
The other two identities are established analogously.
\end{proof}

\begin{prop}\label{cor:bcomm}
The following operators are all bounded on $\Omega^\bullet \otimes_B \F$, and hence uniquely extend to bounded operators on $L^2(\Omega^{\bullet} \otimes_B \F)$: For any $b \in B$, 
\begin{enumerate}
\item  $[\nabla, \lambda(b)]$, \, $[\del_{\F}, \lambda(b)]$, \, $[\adel_{\F}, \lambda(b)]$,
\item $[\nabla^{\dagger}, \lambda(b)]$, \, $[\del^{\dagger}_{\F}, \lambda(b)]$, \, $[\adel^{\dagger}_{\F}, \lambda(b)]$,
\item $[\Delta_{\nabla}, \lambda(b)]$, \, $[\Delta_{\del_{\F}}, \lambda(b)]$, \, $[\Delta_{\adel_{\F}}, \lambda(b)]$.
\end{enumerate}
\end{prop}
\begin{proof}
That the first triple of operators are all bounded follows from Lemma \ref{lem:bcomm} above, together with Proposition \ref{prop:formmult}. For the second triple, we note that $[\nabla^{\dagger},\lambda(b)]$ coincides on $\Omega^{\bullet} \otimes_B \F$ with the adjoint of $-[\nabla,\lambda(b)]$, and so, it is a bounded operator. Boundedness of  $[\partial^{\dagger}_{\F},\lambda(b)]$ and $[\bar{\partial}^\dagger_{\F},\lambda(b)]$ is established similarly. Coming now to the third triple, we see that each operator is a sum of bounded operators, and so, is itself a bounded operator.  Finally, we see that since $\Omega^{\bullet} \otimes_B \F$ is by construction dense in $L^2(\Omega^{\bullet} \otimes_B \F)$, each bounded operator on $\Omega^{\bullet} \otimes_B \F$ admits a unique bounded extension to the whole Hilbert space.
\end{proof}

\subsection{A Remark on Norms}

In this brief subsection, we observe that the norm  induced on $B$ by the embedding $\lambda: \Omega^{\bullet} \to \mathbb{B}(L^2(\Omega^{\bullet} \otimes_B \F))$ is less than or equal to the restriction to $B$ of the reduced norm $\|\cdot\|_{\mathrm{red}}$ of $\mathcal{A}_{\mathrm{red}}$, as defined in Appendix \ref{app:CQGs}.

\begin{prop} 
It holds that 
\begin{align*}
\|b\|_{\mathrm{red}} \geq \|\lambda(b)\|_{\mathrm{op}}  \geq \|b\|_{L^2}, & &  \textrm{for all $b \in B$},
\end{align*}
where $\|\cdot\|_{L^2}$ denotes the norm associated to the inner product of $\Omega^{\bullet} \otimes_B \F$, and $\|\cdot\|_{\mathrm{op}}$ denotes the operator norm of $\mathbb{B}(L^2(\Omega^{\bullet} \otimes_B \F))$.  Thus the restriction of $\lambda$ to $B$ extends to a $*$-algebra homomorphism $\mathcal{B}_{\mathrm{red}} \to \mathbb{B}(L^2(\Omega^{\bullet}) \otimes_B \F)$, where $\mathcal{B}_{\mathrm{red}}$ denotes the closure of $B$ in $\mathcal{A}_{\mathrm{red}}$.
\end{prop}
\begin{proof}
For the first inequality take any $\nu \in \Omega^{\bullet}$, and by abuse of notation denote by $\|\cdot\|_{L^2}$ the norm on $A \, \square_{\pi_B} \Phi(\Omega^{\bullet} \otimes_B \F)$ associated to the inner product $\haar \circ g_{\F,\unit}$. The inequality then follows from the following calculation
\begin{align*}
\|\lambda(b)(\nu \otimes f)\|_{L^2} 
=  \, & \|\unit(b\nu \otimes f)\|_{L^2} \\
=  \, & \|b\nu_{(-1)}f_{(-1)} \otimes [\nu_{(0)} \otimes f_{(0)}]\|_{L^2} \\
=  \, & \|(\lambda_A(b) \otimes \id)(\nu_{(-1)}f_{(-1)} \otimes [\nu_{(0)} \otimes f_{(0)}])\|_{L^2}\\
\leq \, &  \|b\|_{\mathrm{red}} \|\nu_{(-1)}f_{(-1)} \otimes [\nu_{(0)} \otimes f_{(0)}]\|_{L^2} \\
= \, &  \|b\|_{\mathrm{red}} \|\nu \otimes f\|_{L^2}.
\end{align*}
The implied extension of $\lambda$ to a map on $\mathcal{B}_{\mathrm{red}}$ now follows immediately.

For the second inequality, let $f \in \F$ be an element of norm $1$, then 
$$
\|\lambda(b)\|_{\mathrm{op}} \geq \|\lambda(b)(1 \otimes f)\|_{L^2} = \|b \otimes f\|_{L^2}  = \|b\|_{L^2}\|f\|_{L^2} = \|b\|_{L^2},
$$
giving the inequality.
\end{proof}


\section{Closed Operators and Operator Domains}

In this section we turn our attention to unbounded operators constructable from the exterior derivatives and holomorphic structures of a twisted CQH-Hermitian space. In particular, we address questions of closability, essential self-adjointness, and operator domains.

\subsection{Peter--Weyl Maps} \label{subsection:PeterWeyl}

By cosemisimplicity of $A$, the abelian category $\mathrm{mod}^{\pi_B}$ is  semisimple, and so ${}^A_B\textrm{mod}_0$ is semisimple. For any $\mathcal{F} \in  {}^A_B\textrm{mod}_0$, we have the decomposition 
\begin{align*}
\mathcal{F}  \cong  A \, \square_{\pi_B} \Phi(\mathcal{F})   \cong \Big(\bigoplus_{V\in \widehat{A}}\mathcal{C}(V)\Big) \, \square_{\pi_B} \Phi(\mathcal{F})  =  \bigoplus_{V\in \widehat{A}} \mathcal{C}(V) \, \square_{\pi_B}  \Phi(\mathcal{F})  =: \bigoplus_{V\in \widehat{A}} \mathcal{F}_V,
\end{align*}
where $\widehat{A}$ denotes the equivalence classes of irreducible $A$-comodules. We call this the \emph{Peter--Weyl decomposition} of $\mathcal{F}$.

For any $V \in {}^A\mathrm{mod}$, the coalgebra $\mathcal{C}(V)$ is isomorphic as a left $A$-comodule to a direct sum of copies of $V$, see \cite[Proposition 11.8]{KSLeabh} for details. Thus, for any left $A$-comodule map $f:\mathcal{F} \to \mathcal{F}$ it holds that 
\begin{align} \label{eqn:PeterWeylMap}
f(\mathcal{F}_{V}) \sseq \mathcal{F}_{V}, & & \textrm{ for all } V \in \widehat{A}.
\end{align} 
More generally, a \emph{Peter--Weyl map}  $f:\mathcal{F} \to \mathcal{F}$ is a $\mathbb{C}$-linear map satisfying (\ref{eqn:PeterWeylMap}). We now present some properties of the Peter--Weyl decomposition and Peter--Weyl maps in the  CQH-Hermitian setting. The proof is completely analogous to the arguments of  \cite[\textsection 5.2]{MMF3}, and so we omit it. 

\begin{prop} \label{prop:PWProp}
For a CQH-Hermitian space $\mathbf{H} = \{B \subseteq A, \Omega^{\bullet},\Omega^{(\bullet,\bullet)},\sigma\}$, and an Hermitian module $\F$, the Peter--Weyl decomposition of $\Omega^\bullet \otimes_B \F $ is orthogonal with respect to $\langle -, - \rangle_{\F}$. Moreover, for any  Peter--Weyl map $f:\Omega^{\bullet} \otimes_B \F  \to \Omega^{\bullet} \otimes_B \F $, it holds that 
\begin{enumerate}
\item $f$ is adjointable on $\Omega^{\bullet} \otimes_B \F $ with respect to $\langle -,- \rangle_{\F}$, and its adjoint is  a Peter--Weyl map,
\item if $f$ is symmetric with respect to $\langle -,- \rangle_{\F}$, then it is diagonalisable on $\Omega^{\bullet} \otimes_B \F $.
\end{enumerate}
\end{prop}

\subsection{Closability and Essential Self-Adjointness} \label{subsection:CESA}

In this subsection we examine closability and essential self-adjointness for unbounded operators on $\Omega^\bullet \otimes_B \F$. In particular, we show that the unbounded operators $\nabla, \del_{\F}$ and $\adel_{\F}$ are closable, and that the Dirac and Laplacian operators are essentially self-adjoint.

\begin{prop} \label{prop:closability} Every Peter--Weyl map $f: \Omega^\bullet \otimes_B \F \to \Omega^\bullet \otimes_B \F$ 
is closable.
\end{prop}
\begin{proof}
Since $f$ is a Peter--Weyl map, it follows from Proposition \ref{prop:PWProp} that it is adjointable on $\Omega^{\bullet} \otimes_B \F $. Now for any $\alpha, \beta \in \Omega^\bullet \otimes_B \F$, consider the linear functional 
\begin{align*}
\Omega^\bullet \otimes_B \F  =  \mathrm{dom}(f) \to \mathbb{C}, & & \beta \mapsto \langle \alpha, f(\beta) \rangle_{\F}.
\end{align*}
Boundedness of the functional follows from the inequality
\begin{align*}
\left|\langle \alpha, f(\beta) \rangle_{\F}\right| = \left|\langle f^\dagger(\alpha), \beta \rangle_{\F}\right|  \leq  \|f^\dagger(\alpha)\|_{L^2}\|\beta\|_{L^2}.
\end{align*}
Hence $\alpha \in \mathrm{dom}(f^\dagger)$, implying that $ \Omega^\bullet \otimes_B \F \sseq \mathrm{dom}(f^\dagger)$, and consequently  that $\mathrm{dom}(f^\dagger)$ is dense in the Hilbert space $L^2(\Omega^\bullet \otimes_B \F)$. It now follows from  Appendix \ref{appendix:Unbounded} that $f$ is closable.
\end{proof}

Since every comodule map is automatically a Peter--Weyl map, we have the following immediate consequences of the proposition.

\begin{cor} 
Every left $A$-comodule map $f: \Omega^\bullet \otimes_B \F \to \Omega^{\bullet} \otimes_B \F$ is closable.
\end{cor}

\begin{cor} \label{cor:dclosed}
The operators $\nabla, \del_{\F}$, and $\adel_{\F}$ are closable.
\end{cor}
\begin{proof}
Since the calculus and complex structure are, by assumption, covariant, the maps $\nabla, \del_{\F}$, and $\adel_{\F}$ are comodule maps, and hence closable.
\end{proof}

We now prove essential self-adjointness for symmetric comodule maps, and then conclude essential self-adjointness for the twisted Dirac and Laplacian operators of a CQH-Hermitian space.

\begin{prop} \label{prop:CESA}
Every symmetric left $A$-comodule map $f: \Omega^\bullet \otimes_B \F \to \Omega^\bullet \otimes_B \F $ is  diagonalisable on $L^2(\Omega^\bullet \otimes_B \F )$, and moreover, is essentially self-adjoint.
\end{prop}
\begin{proof}
Diagonalisability of $f$ as  an operator on $L^2(\Omega^\bullet \otimes_B \F )$ follows immediately from  Proposition~\ref{prop:PWProp} and our assumption that $f$ is symmetric. That $f$ is symmetric also implies that its eigenvalues  are real. Thus the range of the  operators $f - \mathbf{i } \,\id$ and $f + \mathbf{i } \, \id$ must be equal to $\F \otimes_B  \Om^\bullet$, which is to say, the range of each operator is dense in $L^2(\Omega^\bullet \otimes_B \F)$. It now follows from the results  of Appendix \ref{appendix:Unbounded} that $f$ is essentially self-adjoint.
\end{proof}

\begin{cor} \label{cor:CESA}
The Dirac operators  $D_{\del_{\F}}, D_{\adel_{\F}}$, and $D_{\nabla}$, and the Laplace operators $\DEL_{\del_{\F}}$, $\DEL_{\adel_{\F}}$, and $\Delta_{\nabla}$, are diagonalisable and essentially self-adjoint.
\end{cor}

\begin{rem}
It is interesting to observe that the $*$-map of the calculus $\ast:\Omega^{\bullet} \to \Omega^{\bullet}$ extends to an $\mathbb{R}$-linear map $L^2(\Omega) \to L^2(\Omega)$ and that it restricts to an $\mathbb{R}$-linear isomorphism between the domains of the closures of $D_{\del}$ and $D_{\adel}$. Indeed, for any orthonormal $D_{\del}$-eigenbasis $\{\omega_k\}_k$ of $\Omega^{\bullet}$, an orthonormal  $D_{\adel}$-eigenbasis is given by $\{\omega^*_k\}_k$, where if $\lambda_k$ is the $D_{\del}$-eigenvalue of $\omega_k$, then $\overline{\lambda_k}$ is the $D_{\adel}$-eigenvalue of $\omega_k^*$. Now the domain of the closure of $D_{\del}$ consists of elements $\sum_i c_i \omega_i$ such that $\sum_i |c_i|^2 < \infty$ and  $\sum_i |c_i|^2|\lambda_i|^2 < \infty$, with an analogous description of the domain of the closure of $D_{\adel}$. Thus it is clear that the $*$-map interchanges the two domains.
\end{rem}

\subsection{Cores and Domains}

As recalled in Appendix \ref{app:spectraltriple}, one of the defining requirements of a spectral triple $(A,\H,D)$  is that the domain of the unbounded operator $D$ is closed under the action of $\lambda(a)$, for all $a \in A$. This subtle condition can be verified using cores.  Recall that a \emph{core} for a closable operator $T:\dom(T) \to \H$ is a subset $X \sseq \dom(T)$ such that the closure of $T$ is equal to the closure of the restriction of $T$ to $X$, which is to say,
$
\left(T|_X\right)^c = T^c. 
$
Let $\mathcal{H}$ be a separable Hilbert space,  $D: \mathrm{dom}(D) \sseq \H \to \H$  a densely-defined  closed  operator,  $X \sseq  \mathrm{dom}(D)$ a core for $D$, and $K \in  \mathbb{B}(\H)$ such that $K(X)$ is contained in $\dom(D)$,  and $[D,K]: X \to \H$  is  bounded on $X$. Then, as established by Forsyth, Mesland, and Rennie in \cite[Proposition 2.1]{RennieMesland}, we have that  $K(\dom(D)) \sseq \dom(D)$.

Applying this proposition directly to a general CQH-Hermitian space, we get the following result.

\begin{prop} \label{prop:domains}
Let $\mathbf{H} = (B \subseteq A,\Omega^\bullet,\Omega^{(\bullet,\bullet)},\sigma)$ be a CQH-Hermitian space, $\F$ an Hermitian holomorphic module, and denote by $D_{\adel_{\F}}$ the associated  twisted Dolbeault--Dirac operator. If $A$ is finitely generated as an algebra, then it holds that 
\begin{align*}
\lambda(b)\mathrm{dom}\Big(D_{\adel_\F}\Big) \sseq \mathrm{dom}\Big(D_{\adel_\F}\Big), & & \textrm{ for all } b \in B.
\end{align*}
\end{prop}
\begin{proof}
Since we are assuming that $A$ is finitely generated as an algebra, it follows from Proposition \ref{prop:separable} that $L^2(\Omega^{\bullet} \otimes_B \F)$ is separable. The subspace $\Omega^{\bullet} \otimes_B \F \sseq \mathrm{dom}(D_{\adel_\F})$ is a core  by construction of the 
closure of $D_{\adel_{\F}}$. The core is clearly closed under the action of $\lambda(b)$, for all $b \in B$. Proposition \ref{cor:bcomm} says that $[D_{\adel_{\F}},\lambda(b)]$ is a bounded operator on $\Omega^{\bullet} \otimes_B \F$, for all $b \in B$, and so, we see that $\lambda(b)\mathrm{dom}(D_{\adel_{\F}})$ is contained in $\dom(D_{\adel_{\F}})$ as claimed. 
\end{proof}


\section{Twisted Dolbeault--Dirac Fredholm Operators} \label{section:TFFFred}

In this section we address the Fredholm property for twisted Dolbeault--Dirac operators. More precisely, we show that twisting the Dolbeault--Dirac operator of a CQH-K\"ahler space by a negative Hermitian holomorphic module produces a Fredholm operator if and only if the top anti-holomorphic cohomology group is finite-dimensional. In this case, we also observe that the index of the Fredholm operator is expressible in terms of the dimension of the cohomology group. 

\subsection{The Holomorphic Euler Characteristic}

In this subsection we consider the natural noncommutative generalisation of the  anti-holomorphic Euler characteristic of a classical complex manifold. 

\begin{defn}
Consider a dc  $\Om^\bullet$,  a complex structure $\Om^{(\bullet,\bullet)}$, and an Hermitian holomorphic module $(\F,\adel_{\F})$ with finite-dimensional anti-holomorphic cohomologies. We define the \emph{holomorphic Euler characteristic} of $\F$ to be the value
\begin{align*}
 \chi_{\adel_{\F}}  := \sum_{k \in \mathbb{Z}_{\geq 0}}  (-1)^k \dim\!\left(H^{(0,k)}_{\adel_{\F}}\right) \in \mathbb{Z},
\end{align*}
where we have denoted by $H_{\adel_{\F}}^{(0,k)}$ the $k$-cohomology group of the complex of twisted anti-holomorphic forms  
$$
\adel_{\F}:\Omega^{(0,\bullet)} \otimes_B \F \to \Omega^{(0,\bullet)} \otimes_B \F.
$$
\end{defn}
Note that there exist holomorphic modules with infinite-dimensional holomorphic Euler characteristics, and for these examples the Euler characteristic is not defined.

\subsection{Fredholm Operators}

We begin by recalling the definition of an (unbounded) Fredholm operator, which generalises the index theoretic properties of elliptic differential operators over a compact manifold.

\begin{defn}
For $\mathcal{H}_1$ and $\mathcal{H}_2$ two Hilbert spaces, and $T: \mathrm{dom}(T) \sseq \mathcal{H}_1 \to \mathcal{H}_2$ a densely defined closed linear operator, we say that $T$  is a {\em  Fredholm operator} if $\mathrm{ker}(T)$ and $\mathrm{coker}(T)$ are both finite-dimensional. The \emph{index} of a Fredholm operator $T$ is then defined to be the integer
\begin{align*}
\mathrm{index}(T) := \mathrm{dim}\left(\mathrm{ker}(T)\right) -  \mathrm{dim}\!\left(\mathrm{coker}(T)\right)\!.
\end{align*}
The image $\mathrm{im}(T)$ of a Fredholm operator $T$ is always closed \cite[\textsection 2]{SlechterFred}.
 \end{defn}

\subsection{The Dolbeault--Dirac Fredholm Index} \label{subsection:Fred}

Since $D_{\adel_{\F}}$ is a self-adjoint operator, if it were a Fredholm operator then its index  would necessarily be zero. However, we can alternatively calculate its index with respect to the canonical $\mathbb{Z}_2$-grading of the Hilbert space. For any CQH-Hermitian space, we introduce the spaces 
\begin{align*}
\Omega^{(0,\bullet)}_{\mathrm{even}} \otimes_B \F := \bigoplus_{k \in \mathbb{Z}_{\geq 0}} \Omega^{(0,2k)} \otimes_B \F, & & \Omega^{(0,\bullet)}_{\mathrm{odd}} \otimes_B \F := \bigoplus_{k \in \mathbb{Z}_{\geq 0}} \Omega^{(0,2k+1)} \otimes_B \F,
\end{align*} 
and the associated Hilbert space completions $L^2\!\left(\Omega^{(0,\bullet)}_{\mathrm{even}} \otimes_B \F\right)$ and $L^2\!\left(\Omega^{(0,\bullet)}_{\mathrm{odd}} \otimes_B  \F\right)$.
Define the restricted operator
\begin{align*} 
D_{\adel_{\F}}^+:\dom(D_{\adel_{\F}}) \cap L^2\!\left(\Omega^{(0,\bullet)}_{\mathrm{even}} \otimes_B \F\right) \to L^2\Big(\Om^{(0,\bullet)}_{\mathrm{odd}} \otimes_B  \F\Big), & & x \mapsto D_{\adel_{\F}}(x).
\end{align*}

\begin{prop} \label{prop:index.Euler}
Let $\F$ be an Hermitian holomorphic module with finite-dimensional anti-holomorphic cohomology groups.  If $D_{\adel_{\F}}^+$ is a Fredholm operator,  then its index is equal to the anti-holomorphic Euler characteristic of $\Omega^{(\bullet,\bullet)} \otimes_B \F$, which is to say, 
\begin{align*}
\mathrm{index}\!\left(D_{\adel_{\F}}^+\right) =  \chi_{\adel_{\F}}.
\end{align*}
\end{prop}
\begin{proof}
Let $D_{\adel_{\F}}^+$ be a Fredholm operator, and consider its index
\begin{align*}
\mathrm{index}\!\left(D_{\adel_{\F}}^+\right)  
= & \, \dim\!\left(\ker\!\left(D_{\adel_{\F}}^+\right)\right)  - \dim\!\left(\mathrm{coker}\!\left(D_{\adel_{\F}}^+\right)\right).
\end{align*}
It follows from Hodge decomposition that 
\begin{align*}
\mathrm{index}\!\left(D_{\adel_{\F}}^+\right)   = & \, \sum_{k \in 2\mathbb{Z}_{\geq 0}}  \dim\!\Big(H^{(0,k)}_{\adel_{\F}}\Big) - \sum_{k  \in 2 \mathbb{Z}_{\geq 0} + 1}  \dim\!\Big(H^{(0,k)}_{\adel_{\F}}\Big)
=  \, \sum_{k \in \mathbb{Z}_{\geq 0}}  (-1)^k \textrm{dim}\Big(H^{(0,k)}_{\adel_{\F}}\Big).
\end{align*}
Thus we see that the index of $D_{\adel_{\F}}^+$ is equal to $\chi_{\adel_{\F}}$ as claimed.
\end{proof}

\subsection{Fredholm Operators from Twisting}

In this section we show that twisting the Dolbeault--Dirac operator of CQH-K\"ahler space by a negative Hermitian holomorphic module produces a Fredholm operator if its anti-holomorphic cohomology groups are finite-dimensional. Moreover, in this case the index of the twisted operator is given by the dimension of this cohomology group. The proof combines noncommutative  Hodge decomposition, the noncommutative Kodaira vanishing theorem, and the existence of spectral gaps for negative modules. This gives a perfect example of how the analytic behaviour of the Dolbeault--Dirac operators is shaped by the complex geometry of the underlying dc. This result will be used in \textsection \ref{section:HK} to construct Dolbeault--Dirac Fredholm operators for all the irreducible quantum flag manifolds.

\begin{thm} \label{thm:THETHM}
If $\mathcal{F}$ is a negative module over a $2n$-dimensional CQH-K\"ahler space, with finite-dimensional anti-holomorphic cohomology groups, then the twisted Dirac operator  
\begin{align*} 
D_{\adel_{\mathcal{F}}}^+:\dom(D_{\adel_{\mathcal{F}}}) \cap L^2\!\left(\Omega^{(0,\bullet)}_{\mathrm{even}} \otimes_B \F\right) \to L^2\Big(\Om^{(0,\bullet)}_{\mathrm{odd}} \otimes_B \F\Big)
\end{align*}
is a Fredholm operator.
\end{thm}
\begin{proof}
By the equivalence between cohomology classes and harmonic forms implied by Hodge decomposition, we have that
\begin{align*} 
\mathrm{dim}\!\left(\ker(D_{\adel_{\F}}^+)\right)  = \sum_{k \in 2\mathbb{Z}_{\geq 0}} \mathrm{dim}\!\left(H^{(0,k)}_{\adel_{\F}}\right) < \infty.
\end{align*}
Since $D_{\adel_{\F}}$ is diagonalisable on $\Omega^\bullet \otimes_B \F$, its closure cannot admit an additional non-trivial eigenvector with eigenvalue zero. So in particular, the operator $D^+_{\adel_{\F}}$ and its closure have the same finite-dimensional kernel. 

Let us now move on to the cokernel of the operator. By  \cite[Theorem 3.4]{SpectralGap} we know that the absolute value of the non-zero eigenvalues of $D_{\adel_{\F}}$ are bounded below by a non-zero constant.  Let us now identify 
\begin{align} \label{eqn:L2NonHarm}
L^2\!\left(\del_{\F}(\Omega^\bullet \otimes_B  \F) \oplus \adel^{\dagger}_{\F}(\Omega^\bullet \otimes_B  \F) \right)
\end{align}
with the $\ell^2$-sequences for some choice of basis $\{e_n\}_{n\in \mathbb{Z}_{\geq 0}}$ which diagonalises $D_{\adel_{\F}}$. Taking any such $\ell^2$-sequence $\sum_{n=0}^{\infty} a_n e_n$,  and denoting  $D_{\adel_{\F}}(e_n) =: \mu_n e_n$, we see that 
\begin{align*}
\left\|\sum_{n=0}^{\infty} \mu_n^{-1} a_n e_n \right \|_{L^2} \leq  \sup_{n\in \mathbb{Z}_{\geq 0}} |\mu_n|^{-1}  \left\|\sum_{n=0}^{\infty}  a_n e_n \right\|_{L^2} < \infty.
\end{align*}
Hence $\sum_{n=0}^{\infty} \mu_n^{-1} a_n e_n$ is a well-defined element of $L^2\big( \Omega^{\bullet} \otimes_B  \F \big)$. Moreover, since 
\begin{align*}
D_{\adel_{\F}}\!\!\left(\sum_{n=0}^{\infty} \mu_n^{-1} a_n e_n\right) = \sum_{n=0}^{\infty} a_n e_n,
\end{align*}
we now see that the image of $D_{\adel_{\F}}$ is equal to \eqref{eqn:L2NonHarm}. In particular, appealing again to Hodge decomposition, we see that 
\begin{align*}
\mathrm{dim}\!\left(\mathrm{coker}(D^+_{\adel_{\F}})\right) = \bigoplus_{k \in 2\mathbb{Z}_{\geq 0}+1} \mathrm{dim}\!\left(\mathcal{H}^{(0,k)}\right).
\end{align*}
Thus the cokernel of the operator is finite-dimensional if $\F$ has finite-dimensional odd degree anti-holomorphic cohomologies. Hence we have a Fredholm operator.
\end{proof}


\section{Heckenberger--Kolb Calculi} \label{section:HK}

In this section we present our motivating family of examples, the irreducible quantum flag manifolds endowed with their Heckenberger--Kolb calculi. Using Theorem \ref{thm:THETHM} we show that twisting the Dolbeault--Dirac operator of a Heckenberger--Kolb calculus by negative line modules produces a Fredholm operator. The proof relies on the Borel--Weil theorem for the irreducible quantum flag manifolds \cite{CDOBBW}, which in addition to establishing that the operator is Fredholm, allows us to give an explicit value for its operator index.

\subsection{Drinfeld--Jimbo Quantum Groups}

Let $\frak{g}$ be a finite-dimensional complex simple Lie algebra of rank $r$, and fix a Cartan subalgebra and  a set of simple roots $\Pi = \{\alpha_1, \dots, \alpha_r\}$.  For $q \in \bR_{>0}$ such that  $q \neq 1$, we denote by $U_q(\frak{g})$ the Drinfeld--Jimbo quantised enveloping algebra. We denote the generators by $E_i,F_i,K_i$, for $i = 1, \dots, r$ and follow the conventions of \cite[\textsection 7]{KSLeabh}.  Moreover, we endow $U_q(\frak{g})$ with the compact real form Hopf $*$-algebra structure.

We denote the  fundamental weights  of $\frak{g}$ by $\{\varpi_1, \dots, \varpi_r\}$, and by  $\mathcal{P}^+$ the cone of  {\em dominant integral weights}.  For each $\mu\in\mathcal{P}^+$, we denote by $V_{\mu}$ the corresponding finite-dimensional type-$1$, or admissable, $U_q(\frak{g})$ highest weight module  $V_\mu$. We recall that $V_{\mu}$ has the same dimension as its classical counterpart.
 
Associated to $V$, a finite-dimensional $U_q(\mathfrak{g})$-module, $v \in V$, and $f \in V^*$, the linear dual of $V$, we have the functional
\begin{align*}
c^{\textrm{\tiny $V$}}_{f,v}:U_q(\mathfrak{g}) \to \mathbb{C}, & & X \mapsto f\big(X(v)\big).
\end{align*}
Consider the Hopf subalgebra of $U_q(\mathfrak{g})^\circ$, the Hopf dual of $U_q(\mathfrak{g})$, generated by all functionals of the form $c^{\textrm{\tiny $V$}}_{f,v}$, for $V$ a type-$1$ representation. We denote this Hopf $*$-algebra by $\O_q(G)$ and call it the {\em Drinfeld--Jimbo quantum coordinate algebra of $G$}, where $G$ is the compact, simply-connected, simple Lie group having $\mathfrak{g}$ as its complexified Lie algebra. Note that by construction, $\O_q(G)$ is a CQGA. We denote by $\langle -,-\rangle$ the dual pairing between $U_q(\frak{g})$ by $\O_q(G)$.

\subsection{Quantum Flag Manifolds}

For $S \subset \Pi$ a non-empty subset of simple roots, consider the Hopf $*$-subalgebra
\begin{align*}
U_q(\mathfrak{l}_S) := \big< K_i, E_s, F_s \,|\, i = 1, \dots, r; \, s \in S \big> \subseteq U_q(\mathfrak{g}).
\end{align*} 
We call the \mbox{$*$-subalgebra} 
\begin{align*}
\O_q(G/L_S) := {}^{U_q(\mathfrak{l}_S)}\O_q(G) = \Big\{b \in \O_q(G) \,| \, b_{(1)}\langle X,b_{(1)}\rangle  = \e(X)b, \textrm{ for all } X \in U_q(\frak{l}_S)\Big\},
\end{align*}
the \emph{quantum flag manifold} associated to $S$. Note that $\O_q(G/L_S) $ is a left \mbox{$\O_q(G)$-comodule} algebra by construction. Moreover, $\O_q(G)$ is faithfully flat as a right $\O_q(G/L_S) $-module (see for example \cite[\textsection 5.4]{GAPP}). Thus it follows from \cite[Theorem 1]{Tak} that $\O_q(G/L_S) $ coincides with the space of right coinvariants of the coaction $\Delta_R := (\id \otimes \pi_S) \circ \Delta$, where 
$$
\pi_S: \O_q(G) \to \O_q(L_S) := \O_q(G)/\O_q(G/L_S)^+\O_q(G)
$$
is the canonical Hopf algebra projection. In particular, we note that $\O_q(G/L_S) $ is a CQGA-homogeneous space.

\subsection{Relative Line Modules over the Irreducible Quantum Flag Manifolds} \label{subsection:LBs}

In this subsection we discuss relative line modules over a special subfamily of quantum flag manifolds. If $S = \{\alpha_1, \ldots, \alpha_r\}\backslash \{\a_x\}$, where $\alpha_x$ has coefficient $1$ in the expansion of the longest root of $\frak{g}$, then we say that the associated quantum flag manifold is  {\em irreducible}.

\begin{center}
\begin{table}[ht]
    \caption{\small{Irreducible Quantum Flag Manifolds: organised by series, with defining  crossed node, CQGA-homogeneous space symbol, and name.}} \label{table:CQFMs}
{\small \renewcommand{\arraystretch}{2}%
 \begin{tabular}{|c|c|c|c|}

\hline

\small $A_n$&
\begin{tikzpicture}[scale=.5]
\draw
(0,0) circle [radius=.25] 
(8,0) circle [radius=.25] 
(2,0)  circle [radius=.25]  
(6,0) circle [radius=.25] ; 

\draw[fill=black]
(4,0) circle  [radius=.25] ;

\draw[thick,dotted]
(2.25,0) -- (3.75,0)
(4.25,0) -- (5.75,0);

\draw[thick]
(.25,0) -- (1.75,0)
(6.25,0) -- (7.75,0);
\end{tikzpicture} & \small $\O_q(\text{Gr}_{n+1,s})$ & \small quantum Grassmannian  \\

\small $B_n$ &
\begin{tikzpicture}[scale=.5]
\draw
(4,0) circle [radius=.25] 
(2,0) circle [radius=.25] 
(6,0)  circle [radius=.25]  
(8,0) circle [radius=.25] ; 
\draw[fill=black]
(0,0) circle [radius=.25];

\draw[thick]
(.25,0) -- (1.75,0);

\draw[thick,dotted]
(2.25,0) -- (3.75,0)
(4.25,0) -- (5.75,0);

\draw[thick] 
(6.25,-.06) --++ (1.5,0)
(6.25,+.06) --++ (1.5,0);                      

\draw[thick]
(7,0.15) --++ (-60:.2)
(7,-.15) --++ (60:.2);
\end{tikzpicture} & \small $\O_q(\mathbf{Q}_{2n+1})$ & \small {odd} quantum quadric  \\ 

\small $C_n$& 
\begin{tikzpicture}[scale=.5]
\draw
(0,0) circle [radius=.25] 
(2,0) circle [radius=.25] 
(4,0)  circle [radius=.25]  
(6,0) circle [radius=.25] ; 
\draw[fill=black]
(8,0) circle [radius=.25];

\draw[thick]
(.25,0) -- (1.75,0);

\draw[thick,dotted]
(2.25,0) -- (3.75,0)
(4.25,0) -- (5.75,0);

\draw[thick] 
(6.25,-.06) --++ (1.5,0)
(6.25,+.06) --++ (1.5,0);                      

\draw[thick]
(7,0) --++ (60:.2)
(7,0) --++ (-60:.2);
\end{tikzpicture} &\small   $\O_q(\mathbf{L}_{n})$ & \small 
quantum Lagrangian Grassmannian    \\ 

\small $D_n$& 
\begin{tikzpicture}[scale=.5]

\draw[fill=black]
(0,0) circle [radius=.25] ;

\draw
(2,0) circle [radius=.25] 
(4,0)  circle [radius=.25]  
(6,.5) circle [radius=.25] 
(6,-.5) circle [radius=.25];

\draw[thick]
(.25,0) -- (1.75,0)
(4.25,0.1) -- (5.75,.5)
(4.25,-0.1) -- (5.75,-.5);

\draw[thick,dotted]
(2.25,0) -- (3.75,0);
\end{tikzpicture} &\small   $\O_q(\mathbf{Q}_{2n})$ & \small even quantum quadric  \\

\small $D_n$ & 
\begin{tikzpicture}[scale=.5]
\draw
(0,0) circle [radius=.25] 
(2,0) circle [radius=.25] 
(4,0)  circle [radius=.25] ;

\draw[fill=black] 
(6,.5) circle [radius=.25];
\draw
(6,-.5) circle [radius=.25];

\draw[thick]
(.25,0) -- (1.75,0)
(4.25,0.1) -- (5.75,.5)
(4.25,-0.1) -- (5.75,-.5);

\draw[thick,dotted]
(2.25,0) -- (3.75,0);
\end{tikzpicture} &\small   $\O_q(\textbf{S}_{n})$ & \small quantum spinor variety   \\

\small $E_6$& \begin{tikzpicture}[scale=.5]
\draw
(2,0) circle [radius=.25] 
(4,0) circle [radius=.25] 
(4,1) circle [radius=.25]
(6,0)  circle [radius=.25] ;

\draw
(0,0) circle [radius=.25];
\draw[fill=black] 
(8,0) circle [radius=.25];

\draw[thick]
(.25,0) -- (1.75,0)
(2.25,0) -- (3.75,0)
(4.25,0) -- (5.75,0)
(6.25,0) -- (7.75,0)
(4,.25) -- (4, .75);
\end{tikzpicture}

 &\small  $\O_q(\mathbb{OP}^2)$ & \small  quantum Caley plane   \\
\small $E_7$& 
\begin{tikzpicture}[scale=.5]
\draw
(0,0) circle [radius=.25] 
(2,0) circle [radius=.25] 
(4,0) circle [radius=.25] 
(4,1) circle [radius=.25]
(6,0)  circle [radius=.25] 
(8,0) circle [radius=.25];

\draw[fill=black] 
(10,0) circle [radius=.25];

\draw[thick]
(.25,0) -- (1.75,0)
(2.25,0) -- (3.75,0)
(4.25,0) -- (5.75,0)
(6.25,0) -- (7.75,0)
(8.25, 0) -- (9.75,0)
(4,.25) -- (4, .75);
\end{tikzpicture} &\small   $\O_q(\textbf{F})$ 
& \small   quantum Freudenthal variety   \\
\hline
\end{tabular}
 }
\end{table}
\end{center}

 For the reader's convenience, we recall in the above table the standard pictorial description of the quantum Levi subalgebras defining the irreducible quantum flag manifolds, given in terms of Dynkin diagrams. 

In the irreducible case, the one-dimensional $U_q(\mathfrak{l}_S)$-modules correspond to the elements of $\mathcal{P}_{S^c}$, the $\mathbb{Z}$-module span of the set $\{\varpi_x \,|\, s \in S\}$. This in turn implies that the one-dimensional $\O_q(L_S)$-comodules correspond to the elements of  $\mathcal{P}_{S^c}$. Thus by Takeuchi's equivalence the relative line modules are indexed by the elements of $\mathcal{P}_{S^c}$. For the special case of the irreducible quantum flag manifolds, we see that 
$
\mathcal{P}_{S^c} = \mathbb{Z}\varpi_x.
$
In this case we denote by $\EE_{l}$ the relative line module corresponding to the weight $l\varpi_x$. 

We make some important observations about relative line modules over the irreducible quantum flag manifolds: Firstly, we note that for all $l \in \mathbb{Z}$, we have $(\EE_l)^\ast \cong \EE_{-l}$. Secondly, we note that for each $l \in \mathbb{Z}$, the sesquilinear pairing
\begin{align*}
h_{\EE_l}:\EE_l \times \EE_l \to \O_q(G/L_S) , & & (e_1,e_2) \mapsto e_1^*e_2,
\end{align*}
gives $\EE_l$ the structure of a covariant Hermitian relative line module. Since $\EE_l$ is a simple object in the abelian category 
$$
{}^{~~~~\O_q(G)}_{\O_q(G/L_S)}\textrm{mod}_0,
$$  
we see that $h_{\EE_l}$ is the unique such structure up to positive scalar multiple. Finally, we recall from \cite[Theorem 4.9]{DOKSS} that for $k > 0$, the line module $\EE_k$ is positive with respect to $\kappa$, and $\EE_{-k}$ is negative.

\subsection{Compact Quantum Homogeneous K\"ahler Spaces}

As established in the seminal papers \cite{HK, HKdR}, over any irreducible quantum flag manifold $\O_q(G/L_S) $, there exists a unique finite-dimensional left $\O_q(G)$-covariant $*$-dc
$$
\Omega^{\bullet}_q(G/L_S) \in \, {}^{~~~~\O_q(G)}_{\O_q(G/L_S)}\textrm{mod}_0
$$
of classical total dimension. 
Moreover, as is clear from the Hecknberger--Kolb construction, each  $\Omega^{\bullet}_q(G/L_S)$ comes endowed with an opposite pair of left $\O_q(G)$-covariant complex structures 
\begin{align*}
\Omega^{(\bullet,\bullet)}_q(G/L_S),  & & \textrm{ and  }   & & \overline{\Omega}^{(\bullet,\bullet)}_q(G/L_S),
\end{align*} 
and these are the unique such complex structures for the $*$-dc.  It follows from \cite[Theorem~5.10]{MarcoConj} and \cite[Proposition 5.5]{SpectralGap} that there exists an open interval $I$ around $1$, and a form $\kappa \in \Omega^{(1,1)}$, such that the pair
$$
(\Omega^{(\bullet,\bullet)}_q(G/L_S), \kappa)
$$
is a left $\O_q(G)$-covariant K\"ahler structure, for all $q \in I$. The associated metric $g_{\kappa}$ is positive definite and this uniquely identifies $\kappa$ up to strictly positive real scalar multiple.  Finally, we recall that the closure of the integral of the K\"ahler structure was established in  \cite{SpectralGap}. Collecting together all these results we now arrive at the following theorem.

\begin{thm} \label{thm:PosDefKappa}
For each irreducible quantum flag manifold $\mathcal{O}_q(G/L_S)$,   the quadruple
$$
 \left(\O_q(G/L_S) , \, \Omega^{\bullet}_q(G/L_S), \,\Omega^{(\bullet,\bullet)}_q(G/L_S), \, \kappa \right)
$$
is a CQH-K\"ahler space, for all $q \in I$.
\end{thm}

\subsection{Canonical Bundles}

For each irreducible quantum flag manifold, the space of top holomorphic forms $\Omega^{(M,0)}$ is a line module over $\OO_q(G/L_S)$. We will write $C_S \in \mathbb{Z}_{>0}$ for the integer specified by
\[
\Omega^{(M,0)} \;\cong\; \EE_{-C_S}.
\]
An explicit formula for $C_S$ was obtained in~\cite{DOKSS}, and, as expected, this matches the corresponding classical value. For the reader's convenience, we summarise this information in the table below, together with the associated dimension $M$.

\begin{center}
\begin{table}[ht]
\label{table.table2.qfm}
\caption{\small{Irreducible quantum flag manifolds: notation for the CQGA-homogeneous space, the Heckenberger--Kolb calculus complex dimension, and the identification of the top holomorphic forms with a line module.}} \label{table:CQFMsEk}
{\small \renewcommand{\arraystretch}{1.8}%
\begin{tabular}{|c|c|c|}
\hline
~~ $\O_q(G/L_S)$ ~~ & ~~ $M \coloneqq \dim\!\big(\Omega^{(1,0)}\big)$ ~~ & ~~ Canonical line module $\Omega^{(M,0)}$ ~~ \\
\hline
$\O_q(\mathrm{Gr}_{n+1,s})$ & $s(n\!+\!1\!-\!s)$ & $\EE_{-(n+1)}$ \\
$\O_q(\mathbf{Q}_{2n+1})$ & $2n-1$ & $\EE_{-2n+1}$ \\
$\O_q(\mathbf{L}_{n})$ & $\frac{n(n+1)}{2}$ & $\EE_{-(n+1)}$ \\
$\O_q(\mathbf{Q}_{2n})$ & $2(n-1)$ & $\EE_{-2(n-1)}$ \\
$\O_q(\mathbf{S}_{n})$ & $\frac{n(n-1)}{2}$ & $\EE_{-2(n-1)}$ \\
$\O_q(\mathbb{OP}^2)$ & $16$ & $\EE_{-12}$ \\
$\O_q(\mathbf{F})$ & $27$ & $\EE_{-18}$ \\
\hline
\end{tabular}
}
\end{table}
\end{center}

\subsection{Twisted Dolbeault--Dirac Fredholm Operators}

Since each $\O_q(G)$ is finitely generated as an algebra, Proposition \ref{prop:separable} implies that the Hilbert space of square integrable twisted forms is separable. Every morphism 
$$
f:  \Omega^{\bullet}_q(G/L_S) \otimes_{\O_q(G/L_S)} \F \to  \Omega^{\bullet}_q(G/L_S) \otimes_{\O_q(G/L_S)} \F
$$
in the category of relative Hopf modules extends to a bounded operator on the Hilbert space. In particular, Corollary \ref{cor:sl2rep} implies that the Hilbert space carries a bounded representation of $\frak{sl}_2$.

As shown in \cite[Theorem 4.5]{DOKSS}, for each relative Hopf module $\F$, there exists a unique $\O_q(G)$-covariant holomorphic structure  
$$
\adel_{\F}:\F \to  \Omega^{(0,1)}_q(G/L_S) \otimes_{\O_q(G/L_S)} \F.
$$ 
We denote the associated Chern connection by $\nabla$, and note that it is a bimodule connection \cite[\textsection 6.5]{LCHK}. It now follows from the results of \textsection \ref{subsection:BoundCommutators} that the commutators $[\nabla,\lambda(b)]$, \, $[\del_{\F},\lambda(b)]$, and $[\adel_{\F},\lambda(b)]$ are bounded operators. Moreover, it follows from the results of \textsection \ref{subsection:CESA} that the operators $\nabla, \del_{\F}$, and $\adel_{\F}$ are essentially self-adjoint, as are the corresponding Laplacians.

It was shown in \cite[Theorem 4.9]{DOKSS} that the Hermitian holomorphic module $\EE_{k}$ is positive, and that $\EE_{-k}$ is negative, for any $k \in \mathbb{Z}_{>0}$. The following theorem, one of the main results of the paper, shows that twisting by the negative line modules produces a Fredholm operator. It builds on the Borel--Weil theorem for irreducible quantum flag manifolds, and Serre duality for noncommutative K\"ahler structures.

\begin{thm}
For $k \in \mathbb{Z}_{>0}$, and $q \in I$, the $\EE_{-k}$-twisted Dolbeault--Dirac operator is a  Fredholm operator. Moreover, the index of the operator is given by 
\begin{align} \label{eqn:index}
\mathrm{index}\!\left(D_{\adel_{\EE_{-k}}} \right) = 
\begin{cases}
~ (-1)^M \dim\!\big(V_{(k-C_S)\varpi_x}\big), & \textrm{ if } k \geq C_S,\\
~ 0, & \textrm{ otherwise},
\end{cases}
\end{align}
where $2M$ is the total dimension of the dc, $\Delta^+$ is  the set of positive roots of $\frak{g}$, and $\rho$ is the half-sum of positive roots. In particular, if $k \geq C_S$, then it holds that 
\begin{align} \label{eqn:Weyl.Index}
\mathrm{index}\!\left(D_{\adel_{\EE_{-k}}} \right) = (-1)^M \frac{\prod_{\alpha \in \Delta^+} \big(\rho + (k-C_S)\varpi_x,\alpha \big)}{\prod_{\alpha \in \Delta^+} (\rho,\alpha)}. 
\end{align}
\end{thm}
\begin{proof}
Since $\EE_{-k}$ is a negative line module, Theorem \ref{thm:THETHM} says that the operator is a Fredholm operator  if the anti-holomorphic cohomology groups of $\EE_{-k}$ are finite-dimensional. As shown in \cite{OSV}, the noncommutative Kodaira vanishing theorem for noncommutative K\"ahler structures now implies that 
\begin{align*}
H^{(M,M-b)}_{\adel_{\EE_{k}}} = 0, & &  \textrm{ for all } b < M.
\end{align*}
It now follows from noncommutative Serre duality \cite[\textsection 6.2]{OSV} that 
\begin{align*}
H^{(0,b)}_{\adel_{\EE_{-k}}} = 0, & &  \textrm{ ~~~~ for all } b < M.
\end{align*}
Moreover, it was shown in \cite{OSV} that Serre duality, together with the Borel--Weil theorem for the irreducible quantum flag manifolds \cite[Theorem 6.1]{CDOBBW}, implies that 
$$
H^{(0,M)}_{\adel_{\EE_{-k}}}
$$
is a finite-dimensional irreducible $U_q(\frak{g})$-module. If $k \geq C_S$, then it has highest weight $-w_0((C_S-k)\varpi_x)$, and otherwise it is the zero module. In particular, the anti-holomorphic cohomology groups of $\EE_{-k}$ are finite-dimensional, implying that the $\EE_{-k}$-twisted Dolbeault--Dirac operator is a Fredholm operator as claimed.

The first equality in \eqref{eqn:index} now follows from the index formula given in Proposition \ref{prop:index.Euler}. The equality in \eqref{eqn:Weyl.Index} follows immediately from the Weyl dimension formula for finite-dimensional $U_q(\frak{g})$-modules \cite[\textsection 24.3]{Humph}.
\end{proof}

\begin{eg} 
Consider the special case of quantum projective space $\O_q(\mathbb{CP}^n)$. This is the $A_{n}$-series irreducible quantum flag manifold corresponding to the subset of simple roots  $S = \Pi\backslash \{\varpi_1\}$,  where we have adopted the standard numbering of roots  \cite[\textsection 11.4]{Humph}. For $k = 1, \dots, n$,  the top anti-holomorphic cohomology group 
$$
H^{(0,n)}_{\adel_{\EE_{-k}}}
$$ 
of the relative line module $\EE_{-k}$ vanishes, meaning that the index of its twisted Dirac operator is zero. For the line module $\EE_{-n-1}$, the dimension of 
$$
H^{(0,n)}_{\adel_{\EE_{-n-1}}}
$$
is $1$, meaning that its twisted Dirac operator has index $(-1)^n$. For the line module $\EE_{-n-2}$, the top cohomology group 
$$
H^{(0,n)}_{\adel_{\EE_{-n-2}}}
$$ 
has the dimension of the representation $V_{\varpi_1}$, which is to say, the dimension of the vector space representation of $\frak{sl}_{n+1}$. Thus we see that 
\begin{align*}
\mathrm{index}\!\left(D_{\adel_{\EE_{-n-2}}}\right) = (-1)^n(n+1).
\end{align*}
\end{eg}

\begin{eg}
More generally, consider the quantum $s$-plane Grassmannian $\O_q(\mathrm{Gr}_{n+1,s})$,  for $s=1, \dots , n$, that is to say, the $A_{n}$-series irreducible quantum flag manifold corresponding to $S = \Pi\backslash \{\varpi_s\}$. Again, for $k = 1, \dots, n$, the relative line module $\EE_{-k}$ has trivial anti-holomorphic cohomology group 
$$
H^{(0,s(n+1-s))}_{\adel_{\EE_{-k}}},
$$ 
meaning that the index of the twisted Dirac operator is zero. For the line module $\EE_{-n-1}$, the dimension of 
$$
H^{(0,s(n+1-s))}_{\adel_{\EE_{-n}}}
$$
is $1$, meaning that its twisted Dirac operator has index $(-1)^{s(n+1-s)}$. For the relative line module $\EE_{-n-2}$, the top anti-holomorphic cohomology group has the same dimension as the fundamental representation $V_{\varpi_s}$, which is to say, the dimension of the $s$-fold exterior power of the vector space representation of $\frak{sl}_{n+1}$. Thus we see that 
\begin{align*}
\mathrm{index}\!\left(D_{\adel_{\EE_{-n-2}}}\right) = (-1)^{s(n+1-s)} \binom{n+1}{s}.
\end{align*}  
\end{eg}

\begin{eg}
Next, we consider the quantum spinor variety $\O_q(\mathbf{S}_n)$, which is to say, the $D_n$-series quantum flag manifold corresponding to the subset of simple roots $S = \Pi\backslash \{\alpha_{n}\}$ (or isomorphically $S = \Pi\backslash \{\alpha_{n-1}\}$). For the line module $\EE_{-k}$, when $k = 1, \dots, 2n-3$, the top anti-holomorphic cohomology group 
$$
H^{(0,\frac{n(n-1)}{2})}_{\adel_{\EE_{-k}}}
$$
vanishes, meaning that the index of the associated twisted Dirac operator is zero. Taking the line module $\EE_{-2(n-1)}$, the cohomology group has dimension $1$, meaning that its twisted Dirac operator has index $(-1)^{\frac{n(n-1)}{2}}$. For the line module $\EE_{-2n+1}$, the cohomology group has the same dimension as the representation $V_{\varpi_n}$, one of the two \emph{half-spin representations}. Thus we have that 
\begin{align*}
\mathrm{index}\!\left(D_{\adel_{\EE_{-2n+1}}}\right) = (-1)^{\frac{n(n-1)}{2}}2^{n-1}.
\end{align*}
\end{eg}

\begin{eg}
Finally, we consider the case of the quantum Freudenthal variety $\OO_q(\mathbf{F})$. Here the line module $\EE_{-19}$ has index 
$$
\mathrm{index}\!\left(D_{\adel_{\EE_{-19}}}\right) = - \mathrm{dim}(V_{\varpi_7}) = -56.
$$
We note that the $56$-dimensional representation $V_{\varpi_7}$ is the lowest-dimensional non-trivial $E_7$-module, and as such, it is the $E_7$-analogue of the vector space representation.
\end{eg}

See \cite[Table 5]{OnishchikVinbergLeabh} for a list of explicit dimensions for other distinguished representations.


\appendix

\section{Compact Quantum Groups} \label{app:CQGs}

In this appendix we present two complementary approaches to compact quantum groups. The first is purely Hopf algebraic and due to Koornwinder and Dijkhuizen \cite{KoornDijk}. The second approach is $C^*$-algebraic and due to Woronowicz \cite{WoroCQPGs}. 

\subsection{Compact Quantum Group Algebras}

For $(V,\Delta_L)$ a left $A$-comodule, its space of {\em matrix elements} is the sub-coalgebra 
\begin{align*}
\mathcal{C}(V) : = \mathrm{span}_{\mathbb{C}} \Big\{(\mathrm{id} \otimes f)\Delta_L(v) \,|\, f \in \text{Hom}_{\mathbb{C}}(V,\mathbb{C}), v \in V \Big\} \sseq A.
\end{align*}
A comodule is irreducible if and only if its coalgebra of matrix elements is irreducible, and, for $W$ another left $A$-comodule,  $\mathcal{C}(V) = \mathcal{C}(W)$ if and only if $V$ is isomorphic to $W$. 

Let us now recall the definition of a cosemisimple Hopf algebra, a natural generalisation of the properties of a reductive algebraic group. (See \cite[Theorem 11.13]{KSLeabh} for details.)

A Hopf algebra $A$ is called  {\em cosemisimple} if it admits a (necessarily unique) linear map $\haar:A \to \bC$, called the {\em Haar functional}, such that $\haar(1) = 1$, and 
\begin{align*}
(\id \oby \haar) \circ  \DEL(a) = \haar(a)1, & & (\haar \oby \id) \circ \DEL(a) = \haar(a)1, & & \textrm{ for all } a \in A.
\end{align*}
This is equivalent to having the \emph{Peter--Weyl decomposition} $A \cong \bigoplus_{V\in \widehat{A}} \mathcal{C}(V)$, where summation is over $\widehat{A}$, the set of all equivalence classes of irreducible right $A$-comodules. A {\em compact quantum group algebra}, or a {\em CQGA}, is a cosemisimple Hopf $*$-algebra $A$  such that  the Haar functional $\haar$ is a state. 
Note that for a compact quantum group algebra $A$, the Haar functional is always \emph{faithful}. 
Moreover, $\haar$ is always a $*$-map.

\subsection{Compact Quantum Groups} \label{appendix:SA:CQGs}

Compact quantum group algebras are the algebraic counterpart of Woronowicz's $C^*$-algebraic notion of a compact quantum group \cite{WoroCQPGs}. Every CQGA can be completed to a compact quantum group, and every such completion admits an extension of $\haar$ to a $C^*$-algebraic state. Moreover, every CQG arises as the completion of a CQGA \cite[Theorem 5.4.1]{Timmermann}. Every completion lives between a smallest and a largest completion, analogous to the full and reduced group $C^*$-algebras \cite[\textsection 5.4.2]{Timmermann}.

The completion relevant to this paper is the smallest completion, whose construction we now briefly recall. (See \cite[\textsection 5.4.2]{Timmermann} for a more detailed presentation.) For $\haar$ the Haar functional of a CQGA $A$, an inner product is defined on $A$ by
\begin{align*}
\langle -,- \rangle_{\haar}: A \times A \to \mathbb{C}, & & (a,b) \mapsto \haar(a^*b).
\end{align*}  
Consider now the  faithful $*$-representation $\lambda_A: A \to \mathrm{End}_{\mathbb{C}}(A)$, uniquely defined by $\lambda_A(a)(b) := ab$, where $\mathrm{End}_{\mathbb{C}}(A)$ denotes the $\mathbb{C}$-linear operators on $A$. For all $a \in A$, the operator $\lambda_A(a)$ is bounded with respect to $\langle -,- \rangle_{\haar}$. Hence, denoting by $L^2(A)$ the associated Hilbert space completion of $A$, each operator $\lambda_A(a)$ extends to an element of $\mathbb{B}(L^2(A))$.  We denote by $\mathcal{A}_{\mathrm{red}}$  the corresponding closure of $\lambda_A(A)$ in $\mathbb{B}(L^2(A))$. The coproduct of $A$ extends to  a $*$-homomorphism $\Delta:\mathcal{A}_{\mathrm{red}} \to \mathcal{A}_{\mathrm{red}} \otimes_{\mathrm{min}} \mathcal{A}_{\mathrm{red}}$, and the pair $(\mathcal{A}_{\mathrm{red}},\Delta)$ forms a CQG.


\section{The Rudiments of Unbounded Operators} \label{appendix:Unbounded}

In this appendix, we present the rudiments of the theory of unbounded operators on Hilbert spaces, with a view to making the paper more accessible to those coming from an algebraic or geometric background. For more details we refer the reader to the standard texts \cite{Rudin} and \cite{HigsonRoe}.

Let $T:\mathrm{dom}(T) \to \mathcal{H}$ be a not necessarily bounded operator on a Hilbert space $\mathcal{H}$, with $\mathrm{dom}(T)$ denoting its domain of definition. We say that $T$ is \emph{closed} if its graph $\mathcal{G}(T)$ is closed in the direct sum $\mathcal{H} \oplus \mathcal{H}$. We say that an operator $T$ is {\em closable} if the closure of its graph in $\mathcal{H} \oplus \mathcal{H}$ is the graph of a (necessarily closed) operator $T^c$, which we call the \emph{closure} of $T$. When no confusion arises we will not distinguish notationally between an operator and its closure.

For $T:\mathrm{dom}(T) \to \mathcal{H}$ a densely-defined operator, the associated {\em adjoint operator} $T^{\dagger}$ has domain consisting of those elements $x \in \mathcal{H}$ such that
\begin{align*}
\psi_x: \mathrm{dom}(T) \to \mathbb{C}, & & y \mapsto \langle x, T(y) \rangle
\end{align*}
is a continuous linear functional. By the Riesz representation theorem, there exists a unique $z \in \mathcal{H}$, such that 
$
\langle z,y \rangle  =  \langle x, T(y) \rangle,  \text{ for all }  y \in \mathrm{dom}(T).
$
The operator $T^{\dagger}$ is then defined as
\begin{align*}
T^{\dagger}: \mathrm{dom}(T^{\dagger}) \to \mathcal{H}, & & x \mapsto z.
\end{align*}
Any  operator whose adjoint is densely-defined is necessarily closable, see \cite[Theorem 13.8]{Rudin} for details.

A densely-defined operator $T$  is said to be \emph{symmetric} if 
\begin{align*}
\langle T(x), y \rangle = \langle x, T(y) \rangle, & & \text{ for all } x,y \in \mathrm{dom}(T).
\end{align*}
An operator $T$ is said to be  \emph{self-adjoint} if it is symmetric and $\mathrm{dom}(T) = \mathrm{dom}(T^{\dagger})$, and is said to be \emph{essentially self-adjoint} if it is closable and its closure is self-adjoint. As explained in \cite[\textsection 13.20]{Rudin}, a densely-defined symmetric operator is essentially self-adjoint if the operators  $T + \mathbf{i} \, \id_{\mathcal{H}}$ and $T - \mathbf{i} \, \id_{\mathcal{H}}$  have dense range.

A complex number $\lambda$ is said to be in the \emph{resolvent set} $\lambda(T)$ of an unbounded operator $T: \dom(T) \to \mathcal{H}$, if 
\begin{align*}
T - \lambda\, \id_{\mathcal{H}}: \dom(T) \to \H, 
\end{align*}
has a bounded inverse, that is,  if there exists a bounded operator
$
 S : \mathcal{H} \to \dom(T) 
$
such that
$
S \circ (T- \lambda \, \id_{\mathcal{H}})= \id_{\dom(T)}
$
and 
$(T- \lambda \, \id_{\mathcal{H}}) \circ S = \id_{\mathcal{H}}$.
The \emph{spectrum} of $T$, which we denote by $\s(T)$,  is the complement of $\lambda(T)$ in $\mathbb{C}$.  Just as in the bounded case, self-adjoint operators have real spectrum. We denote the set of eigenvalues of $T$ by $\sigma_P(T)$ and call it the \emph{point spectrum} of $T$. It is clear from the definition of the spectrum that  $\sigma_P(T) \sseq \sigma(T)$.

We now recall the functional calculus for unbounded self-adjoint operators: For any self-adjoint operator  $T$, and any bounded Borel function $f: \sigma(T) \to \mathbb{C}$, one can associate a bounded operator $f(T):\H \to \H$. This extends the usual functional calculus for bounded operators (see \cite[\textsection 1.8]{HigsonRoe} for details).


\section{Spectral Triples} \label{app:spectraltriple}

In this appendix we recall the definition of a spectral triple and produce sufficient and necessary conditions on the point spectrum of the Dolbeault--Dirac operator of a CQH-K\"ahler space to give a spectral triple. We also discuss how non-vanishing of the anti-holomorphic Euler characteristic of the underlying complex structure implies non-triviality of the associated $K$-homology class.

\subsection{Spectral Triples and the Bounded Transform}

The $K$-homology of a $C^*$-algebra is the unitary equivalence classes of even Fredholm modules up to operator homotopy. In practice the calculation of the index of a $K$-homology class, or more generally its pairing with $K$-theory, can prove difficult. However, the work of Baaj and Julg \cite{BaajJulg}, and Connes and Moscovici \cite{ConnesMosc}, shows that by considering spectral triples, unbounded representatives of $K$-homology classes, the problem can often become more tractable.

\begin{defn}
A {\em spectral triple} $(A,\H,D)$  consists of a unital \mbox{$*$-algebra} $A$, a separable Hilbert space $\H$, endowed with a faithful $*$-representation $\lambda:A \to \mathbb{B}(\H)$, and  \mbox{$D: \dom(D)  \to \H$} a densely-defined self-adjoint operator, such that
\begin{enumerate}
\item $\lambda(a)\text{dom}(D) \sseq \text{dom}(D)$, for all $a \in A$, 

\item $[D,\lambda(a)]$ is a bounded operator,  for all $a \in A$,

\item $(D^2 + \mathbf{i})^{-1} \in \mathcal{K}(\H)$, where $\mathcal{K}(\H)$ denotes the compact operators on $\H$.

\end{enumerate}

An {\em even spectral triple} is a quadruple  $(A,\mathcal{H},D,\gamma)$, consisting of a spectral triple $(A,\mathcal{H},D)$, and a $\bZ_2$-grading $\mathcal{H} = \mathcal{H}_0 \oplus \mathcal{H}_1$ of Hilbert spaces $\gamma$ , with respect to which $D$ is a degree $1$ operator, and $\lambda(a)$ is a degree $0$ operator, for each $a \in A$.
\end{defn}

One reason why spectral triples are important is that they provide unbounded representatives for $K$-homology classes. For a spectral triple $(A,\mathcal{H},D)$, its {\em bounded transform} is the  operator
\begin{align*}
\frak{b}(D) := \frac{D}{\sqrt{1+D^2}} \in \mathbb{B}(\mathcal{H}),
\end{align*}
defined via the functional calculus. 
A Fredholm module is given by $(\mathcal{H}, \lambda, \frak{b}(D))$. (See  \cite{CP1} for details.)  The index of the Fredholm operator $D^+:H_0 \to \mathcal{H}_1$ is clearly equal to the index of the bounded transform. Since the index is an invariant of $K$-homology classes, a spectral triple with  non-zero index has a non-trivial associated $K$-homology class

\subsection{Spectral Triples and Dolbeault--Dirac Eigenvalues}

We now formulate precise criteria for when the Dolbeault--Dirac operator of a CQH-Hermitian space gives a spectral triple. For sake of clarity and convenience, let us recall the relevant properties of $L^2(\Omega^{\bullet})$ and $D_{\adel}$. If $A$ is finitely generated, then it follows from Proposition \ref{prop:separable} that $L^2(\Omega^{\bullet})$ is separable. By Proposition \ref{prop:star.map} we have a faithful $*$-representation $\lambda: B \to \mathbb{B}(L^2(\Omega^\bullet))$. From Corollary \ref{cor:CESA}  we know that  $D_{\adel}$ is an essentially self-adjoint operator, which is, moreover, densely-defined by construction. By Proposition \ref{cor:bcomm}, the commutators $[D_{\adel},\lambda(b)]$ are bounded, and by Proposition  \ref{prop:domains} above, $\lambda(b)\mathrm{dom}(D_{\adel}) \sseq \mathrm{dom}(D_{\adel})$, for all $b \in B$. With respect to the obvious $\mathbb{Z}_2$-grading $\gamma$, the operator $D_{\adel}$ is of degree $1$, and $\lambda(b)$ is a degree $0$ operator, for all $b \in B$. Finally, we note that since  $D_{\adel}$ is diagonalisable on $L^2(\Omega^\bullet)$, it has compact resolvent if and only if its eigenvalues tend to infinity and have finite multiplicity. Collecting these facts together gives the following proposition.

\begin{prop} \label{prop:spectraltrip}
Let $\mathbf{H} = (B \subseteq A,\Omega^{\bullet},\Omega^{(\bullet,\bullet)},\sigma)$ be a CQH-Hermitian space for which $A$ is finitely generated as an algebra, then an even spectral triple is given by 
\begin{align*}
\left(B \subseteq A, L^2(\Om^{(0,\bullet)}), D_{\adel},\gamma\right),
\end{align*}
if and only if the eigenvalues of $D_{\adel}$ tend to infinity and have finite multiplicity.  
\end{prop}

We call such a spectral triple the \emph{Dolbeault--Dirac spectral triple} of $\mathbf{H}$. In the accompanying paper \cite{DOS1}, these criteria were verified for the special case of quantum projective space $\O_q(\mathbb{CP}^n)$, producing a motivating family of examples of Dolbeault--Dirac spectral triples. 

The discussions in \textsection \ref{subsection:Fred} give the following immediate result, where we denote by $\mathcal{B}$ the closure of $\lambda(b)$ in $\mathbb{B}\!\left(L^2(\Omega^{(0,\bullet)})\right)$.

\begin{cor} \label{cor:NTEulermeansNTKClass}
Let $\mathbf{H} = (B \subseteq A,\Omega^{\bullet},\Omega^{(\bullet,\bullet)},\sigma)$ be a CQH-Hermitian space  with a Dolbeault--Dirac spectral triple. The  $K^0(\B)$-class of the spectral triple is non-trivial if the holomorphic Euler characteristic of $\Omega^{(\bullet,\bullet)}$ is non-trivial.
\end{cor}

The notion of a \emph{noncommutative Fano structure} was introduced in \cite[Definition 8.8]{OSV}. It is a refinement of  a K\"ahler structure, generalising the classical definition of a Fano manifold. A \emph{CQH-Fano space} is a CQH-K\"ahler space whose constituent K\"ahler structure is a Fano structure. It follows from Theorem \ref{thm:PosDefKappa} and \cite[Theorem 4.12]{DOKSS} that the irreducible quantum flag manifolds give CQH-Fano structures.

\begin{cor} 
Let $\mathbf{F} = (B \subseteq A,\Omega^{\bullet},\Omega^{(\bullet,\bullet)},\sigma)$ be a CQH-Fano space with a Dolbeault--Dirac spectral triple. Then the  $K^0(\B)$-class of the spectral triple is non-trivial.
\end{cor}
\begin{proof}
It follows from \cite[Corollary 8.9]{OSV} that $H^{(0,k)}_{\adel} = 0$, for all $k >0$. Thus we see that the anti-holomorphic Euler characteristic of the calculus is equal to the dimension of $H^{(0,0)}_{\adel}$. However, since $1$ is always contained in  $\ker(\adel)$, this is always non-zero. Thus it follows from Corollary \ref{cor:NTEulermeansNTKClass} that the $K^0(\B)$-class of the spectral triple is non-trivial.
\end{proof}


We finish this subsection with an easy observation about the Dolbeault--Dirac operator of the opposite CQH-Hermitian space. 

\begin{prop} \label{lem:unitaryequivalence}
For a CQH-Hermitian space $\mathbf{H} = (B,\Omega^{\bullet},\Omega^{(\bullet,\bullet)},\sigma)$,  the two operators $D_{\del}:\Omega^{(\bullet,0)} \to \Omega^{(\bullet,0)}$ and $D_{\adel}:\Omega^{(0,\bullet)} \to \Omega^{(0,\bullet)}$ are unitarily equivalent. In particular, 
\begin{align} \label{eqn:oppST}
 \left(B, L^2(\Om^{(\bullet,0)}), D_{\del}\right)
\end{align}
is a spectral triple if and only if $\left(B, L^2(\Om^{(0,\bullet)}), D_{\adel}\right)$ is a spectral triple. 
\end{prop}
\begin{proof}
A form $\omega^* \in \Omega^{(0,\bullet)}$ is an eigenvector of $D_{\del}$ if and only if $\omega \in \Omega^{(\bullet,0)}$ is an eigenvector of $D_{\adel}$, as we see from the identity
$$
D_{\del}(\omega^*) = D_{\adel}(\omega)^* = (\lambda \omega)^* = \overline{\lambda}\omega^* = \lambda \omega^*.
$$
Thus the set of eigenvalues of $D_{\del}$ coincides with the set of eigenvalues of $D_{\adel}$, and  we have a real linear isomorphism between the respective eigenspaces. Since the eigenspaces of each operator are necessarily orthogonal, we can now construct a unitary map $U:\Omega^{(0,\bullet)} \to \Omega^{(\bullet,0)}$ satisfying $D_{\del} = U \circ D_{\adel} \circ U^{-1}$. Extending $U$ to the domain of the closure of $D_{\adel}$ gives the required unitary equivalence.  It now follows from Proposition~\ref{prop:spectraltrip}  that if one triple is a spectral triple then so is the other.
\end{proof}

It is important to note that the unitary equivalence between the operators $D_{\del}$ and $D_{\adel}$ will not in general be a module map, nor an $A$-comodule map.


\bibliographystyle{abbrv}

\end{document}